\def\todaysdate{9\textsuperscript{th} June 2021}
\documentclass[reqno,a4paper]{article}
\usepackage{etex}
\usepackage[T1]{fontenc}
\usepackage[utf8]{inputenc}
\usepackage{lmodern}
\usepackage{subcaption}

\usepackage{amssymb,amsmath,amsthm}
\usepackage{thmtools,xcolor}
\usepackage{units}
\usepackage{graphicx}
\usepackage[all]{xy}
\usepackage{tikz}
\usetikzlibrary{arrows,calc,positioning,decorations.pathreplacing}
\usepackage{relsize}
\usepackage{mhsetup}
\usepackage{mathtools}
\usepackage{float}
\usepackage{stmaryrd}
\usepackage{tocloft}
\usepackage{paralist} 
\usepackage[left=3cm,right=3cm,top=3cm,bottom=3cm]{geometry} 
\usepackage[bottom]{footmisc}
\usepackage[colorlinks,citecolor=blue,linkcolor=blue,urlcolor=blue,filecolor=blue,breaklinks]{hyperref}
\usepackage{footnotebackref}
\usepackage[format=plain,indention=1em,labelsep=quad,labelfont={up},textfont={footnotesize},margin=2em]{caption}
\usepackage[mathscr]{euscript}
\usepackage{accents}

\definecolor{lightblue}{rgb}{0.8,0.8,1}

\numberwithin{equation}{section}
\numberwithin{figure}{section}

\definecolor{vdarkred}{rgb}{0.7,0,0}
\declaretheoremstyle[
  spaceabove=\topsep,
  spacebelow=\topsep,
  headpunct=,
  numbered=no,
  postheadspace=1ex,
  headfont=\color{vdarkred}\normalfont\bfseries,
  bodyfont=\normalfont\itshape,
]{colored}
\declaretheoremstyle[
  spaceabove=\topsep,
  spacebelow=\topsep,
  headpunct=,
  numbered=no,
  postheadspace=1ex,
  headfont=\normalfont\bfseries,
  bodyfont=\normalfont\itshape,
]{italic}
\declaretheoremstyle[
  spaceabove=\topsep,
  spacebelow=\topsep,
  headpunct=,
  numbered=no,
  postheadspace=1ex,
  headfont=\normalfont\bfseries,
  bodyfont=\normalfont\upshape,
]{upright}

\declaretheorem[style=italic,name=Theorem,numbered=yes,numberwithin=section]{thm}
\declaretheorem[style=italic,name=Lemma,numbered=yes,numberlike=thm]{lem}
\declaretheorem[style=italic,name=Proposition,numbered=yes,numberlike=thm]{prop}
\declaretheorem[style=italic,name=Corollary,numbered=yes,numberlike=thm]{coro}

\declaretheorem[style=upright,name=Definition,numbered=yes,numberlike=thm]{defn}
\declaretheorem[style=upright,name=Remark,numbered=yes,numberlike=thm]{rmk}

\declaretheorem[style=upright,name=Notation,numbered=yes,numberlike=thm]{notation}

\makeatletter
\renewcommand*{\@seccntformat}[1]{\upshape\csname the#1\endcsname.\hspace{1ex}}
\renewcommand*{\section}{\@startsection{section}{1}{\z@}%
	{2.5ex \@plus 1ex \@minus 0.2ex}%
	{1.5ex \@plus 0.2ex}%
	{\normalfont\normalsize\bfseries}}
\renewcommand*{\subsection}{\@startsection{subsection}{2}{\z@}%
	{2.5ex \@plus 1ex \@minus 0.2ex}%
	{-1.5ex \@plus -0.2ex}%
	{\normalfont\normalsize\bfseries}}
\renewcommand*{\subsubsection}{\@startsection{subsubsection}{3}{\z@}%
	{2.5ex \@plus 1ex \@minus 0.2ex}%
	{-1.5ex \@plus -0.2ex}%
	{\normalfont\normalsize\bfseries}}
\renewcommand*{\paragraph}{\@startsection{paragraph}{4}{\z@}%
	{2.5ex \@plus 1ex \@minus 0.2ex}%
	{-1.5ex \@plus -0.2ex}%
	{\normalfont\normalsize\bfseries}}
\renewcommand*{\subparagraph}{\@startsection{subparagraph}{5}{\z@}%
	{2.5ex \@plus 1ex \@minus 0.2ex}%
	{-1.5ex \@plus -0.2ex}%
	{\normalfont\normalsize\slshape}}
\makeatother

\newcommand{\N}{\mathbb N}
\newcommand{\Z}{\mathbb Z}

\definecolor{dgreen}{RGB}{0,150,0}

\newcommand{\incl}[3][right]%
{%
\draw[<-,>=#1 hook] #2 to ($ #2!0.5!#3 $);
\draw[->,>=stealth'] ($ #2!0.5!#3 $) to #3;%
}
\newcommand{\inclusion}[5][right]%
{%
\draw[<-,>=#1 hook] #4 to ($ #4!0.5!#5 $) node[#2,font=\small]{#3};
\draw[->,>=stealth'] ($ #4!0.5!#5 $) to #5;%
}

{\begin{compactitem}

}%
{\end{compactitem}}
{\begin{compactitem}[#1]

}%
{\end{compactitem}}
{\begin{compactdesc}

}%
{\end{compactdesc}}

\newcommand{\CC}{(C_1,...,C_n)}

\newcommand{\cC}{\mathcal{C}}
\newcommand{\cD}{\mathcal{D}}

\newcommand{\cN}{\mathcal{N}}

\newcommand{\cU}{\mathcal{U}}

\renewcommand{\geq}{\geqslant}
\renewcommand{\leq}{\leqslant}
\renewcommand{\footnoterule}{%
  \kern -3pt
  \hrule width \textwidth height 0.4pt
  \kern 2.6pt
}

\definecolor{dgreen}{RGB}{0,150,0}

\begin{document}
\title{\vspace{-12mm} \Large\bfseries Witten-Reshetikhin-Turaev invariants for 3-manifolds from Lagrangian intersections in configuration spaces }
\author{ \small Cristina Anghel \quad $/\!\!/$\quad \todaysdate\vspace{-3ex}}
\date{}
\maketitle
{
\makeatletter
\renewcommand*{\BHFN@OldMakefntext}{}
\makeatother
}
\vspace{-8mm}
\begin{abstract}
In this paper we construct a topological model for the Witten-Reshetikhin-Turaev invariants for $3$-manifolds coming from the quantum group $U_q(sl(2))$, as graded intersection pairings of homology classes in configuration spaces. More precisely, for a fixed level $\cN \in \N$ we show that the level $\cN$ WRT invariant for a $3-$manifold is a state sum of Lagrangian intersections in a covering of a {\bf fixed} configuration space in the punctured disk. This model brings a new perspective on the structure of the level $\cN$ Witten-Reshetikhin-Turaev invariant, showing that it is completely encoded by the intersection points between certain Lagrangian submanifolds in a fixed configuration space, with additional gradings which come from a particular choice of a local system. This formula provides a new framework for investigating the open question about categorifications of the WRT invariants. 
\end{abstract}
\vspace{-5.15mm}
{\tableofcontents \vspace{-2mm}}

\vspace{-2mm}
\section{Introduction}\label{introduction}
\vspace{-3mm}
 After the discovery of the Jones polynomial for knots, the world of quantum invariants encountered a powerful development, provided by constructions due to Witten, Reshetikhin and Turaev. More precisely Witten \cite{Witt} predicted the existence of an extension of the Jones polynomial to $3$-manifolds and Reshetikhin-Turaev \cite{RT} provided an algebraic construction of such invariants.  They showed that the representation theory of the quantum group $U_q(sl(2))$ leads to invariants for links coloured with finite dimensional representations of this quantum group, called coloured Jones polynomials. Further on, for any level $\cN \in \N$, one can use linear combinations of coloured Jones polynomials with colours less than $\cN$ in order to get a $3$-manifold invariant $\tau_{\cN}$. However, there are open questions about the geometry and topology which is contained in the Witten-Reshetikhin-Turaev invariants. An active research area concerns categorifications for invariants of links and $3-$manifolds. For instance, Khovanov homology, which is a categorification for the Jones polynomial for knots, was proved to be a powerful tool which contains much information (\cite{K},\cite{Ras},\cite{KM},\cite{OZ},\cite{SM},\cite{M1}). The story is different for the analogous invariants for $3$-manifolds. There is an important open question about the existence of categorifications for Witten-Reshetikhin-Turaev invariants.

Our aim is to describe these invariants as intersection pairings between homology classes in coverings of configuration spaces. We refer to such a description as ``topological model''. The main result of the paper shows that the level $\cN$ WRT invariant is a state sum of graded intersections between Lagrangian submanifolds in a fixed configuration space. This provides a new framework for the study of these invariants and a starting point in investigating categorification questions. 

In the first part of this article, Theorem \ref{THEOREMM}, we generalise the author's previous work (\cite{Cr3},\cite{Cr5}) constructing a topological model for coloured Jones polynomials coloured with different colours. Then, the translation of the algebraic definition of the WRT invariant using Theorem \ref{THEOREMM} would show that the WRT invariant $\tau_{\cN}$ is a linear combination of Lagrangian intersections in various configuration spaces. Further on, the main part of the paper is geometric. We encode the coefficients of the coloured Jones polynomials coming from the Kirby colour by adding certain circles to the supports of the Lagrangian submanifolds as well as adding extra punctures to the punctured disk. Then, we show that we can move the whole intersection formula -- which a priori would be in different configuration spaces -- in a fixed configuration space, as presented in Theorem \ref{THEOREMW}. 
\subsection{Homological tools} For $n,m \in \N$, we define $C_{n,m}=Conf_m(\mathscr D_n)$ to be the unordered configuration space of $m$ points in the $n$-punctured disc $\mathscr D_n$. We use two extra parameters $k,\bar{l}\in \N$ and define a local system:
$$\Phi: \pi_1(C_{n+3\bar{l},m}) \rightarrow \Z^n \oplus \Z^{\bar{l}} \oplus \Z.$$
 The definition of this local system depends on the parameter $k$. Roughly speaking, the monodromy around each puncture gives us one variable and the last $\Z-$component counts the winding of particles in the configuration space. The parameter $k$ is used for orientation purposes: the monodromies  of $\Phi$ around the first $n-k$ punctures and the last $k$ punctures are counted with opposite orientations. In our model $\bar{l}$ will be the number of link components. The extra $3\bar{l}$ punctures will play an important role in the model for the WRT-invariants. We define $\tilde{C}^{-k}_{n+3\bar{l},m}$ to be the covering of $C_{n,m}$ corresponding to $\Phi$. 
 We use the homology of a quotient of this covering space (quotienting the first $n$ components of the local system towards $l$ variables, for $l \leq n$), as follows: 
\vspace{-2mm}
\begin{enumerate}
\item[•] Lawrence representations $H^{-k}_{n,m,\bar{l}}$ which are $\Z[x_1^{\pm 1},...,x_l^{\pm 1},y_1^{\pm 1}..., y_{\bar{l}}^{\pm 1}, d^{\pm 1}]$-modules.\\
They come from the Borel-Moore homology of $\tilde{C}^{-k}_{n+3\bar{l},m}$ and have an action of coloured braids on $n+3\bar{l}$ strands (Definition \ref{D:4}, Proposition \ref{colbr})
\item[•] Dual Lawrence representations $H^{-k,\partial}_{n,m,\bar{l}}$ (Definition \ref{D:4})\\ 
(using the homology relative to the boundary of the same covering space) 
\item[•] Graded intersection pairing (Proposition \ref{P:3'}):
\begin{equation*}
\left\langle , \right\rangle:H^{-k}_{n,m,\bar{l}} \otimes H^{-k,\partial}_{n,m,\bar{l}}\rightarrow \Z[x_1^{\pm 1},...,x_l^{\pm 1},y_1^{\pm 1}..., y_{\bar{l}}^{\pm 1}, d^{\pm 1}].
\end{equation*}
\end{enumerate}
\vspace{-2mm}
{\bf Homology classes} We will construct certain classes in these homology groups, given by lifts of Lagrangian submanifolds in the base configuration space. These submanifolds are encoded by ``geometric supports'' which are sets of arcs in the punctured disc. The product of these arcs quotiented to the unordered configuration space gives the Lagrangian submanifolds. Then, the lifts in the covering will be encoded by sets of ``paths to the base points'' which are collections of arcs in the punctured disk, from the base point towards the geometric support. 
\begin{rmk}\label{pairing} \ \ 
The pairing is encoded in the base configuration space, and it is parametrised by the intersection points between the geometric supports of the homology classes, graded by monomials which are prescribed by the local system $\Phi$.
\end{rmk}
\subsection{Topological model coloured Jones polynomials}
 In the author's earlier work, \cite{Cr3} and \cite{Cr5}, a topological model for the coloured Jones polynomial for links coloured with the same colour was constructed (i.e. each component is coloured with the same colour). In the first part of this paper we generalize this result and construct a topological model for coloured Jones polynomials for links coloured with different colours. 
Let $L$ be an oriented framed link with framings $f_1,...,f_l\in \Z$. We consider $\beta_n \in B_n$ a braid such that $L= \widehat{\beta_n}$ by braid closure. 
Now, let us fix a set of colours $N_1,...,N_l\in \N$ for the strands of the link. This colouring induces a colouring of the strands of the braid:
$(C_1,...,C_n).$

We use the configuration space of $1+\sum_{i=1}^{n} (C_i-1)$ particles in the $(2n+1)$-punctured disk, and a $\mathbb Z^{2n+1}\oplus \Z$ local system constructed as above, with $k=n$ and $\bar{l}=0$. Then, we have the homologies:
$$H^{-n}_{2n+1, {1+\sum_{i=1}^{n}}(C_i-1),0} \text { and } H^{-n,\partial}_{2n+1, {1+\sum_{i=1}^{n}}(C_i-1),0} \text{ which are } \Z[x_1^{\pm 1},...,x_l^{\pm 1}, d^{\pm 1}] \text{-modules.}$$
\begin{defn}(Coloured Homology classes)
With the procedure described above, for any indices $i_1,...,i_{n}\in \N$ such that $0 \leq i_k \leq C_k-1$ for all $k\in \{1,...,n\}$ we define two Lagrangian submanifolds and consider the classes given by their lifts in the covering, as presented in figure \ref{Picture0}:
$${\color{red} \mathscr F_{\bar{i}}^{\CC} \in H^{-n}_{2n+1, {1+\sum_{i=1}^{n}}(C_i-1),0}} \ \ \ \ \ \ \ \ \text{ and }\ \ \ \ \ \ \ \ \  {\color{dgreen} \mathscr L^{\CC}_{\bar{i}}\in H^{-n,\partial}_{2n+1,1+\sum_{i=1}^{n} (C_i-1),0}}.$$
The set of such sequences of indices is denoted by $C(\bar{N})$.
\end{defn}
\begin{thm}\label{THEOREMM}(Topological state sum model for coloured Jones polynomials for coloured links)\\
Let us fix a set of colours $N_1,..,N_l \in \N$. Then, the coloured Jones polynomial of $L$ coloured with colours $N_1,...,N_l$ has the following model:
\begin{equation}
\begin{aligned}
J_{N_1,...,N_l}(L,q)& =~ q^{ \sum_{i=1}^{l}\left( f_i- \sum_{j \neq i} lk_{i,j}\right)(N_i-1)} \cdot \\
 & \cdot \left(\sum_{\bar{i}\in C(\bar{N})} \left( \prod_{i=1}^{n}x^{-1}_{C(i)} \right)\cdot  \left\langle(\beta_{n} \cup {\mathbb I}_{n+1} ) \ { \color{red} \mathscr F_{\bar{i}}^{\CC}}, {\color{dgreen} \mathscr L_{\bar{i}}^{\CC}}\right\rangle \right)\Bigm| _{\psi^{C}_{q,N_1,...,N_l}}.
\end{aligned}
\end{equation} 
In this expression $\psi^{C}_{q,N_1,...,N_l}$ is the specialisation of variables to one variable from formula \eqref{eq:8''''} and $C:\{1,...,2n+1\}\rightarrow \{1,...,l\}$ is the colouring presented in equation \eqref{eq:col} and remark \ref{not'}.

\end{thm}
Note that this formula is a state sum of intersections in a configuration space where the number of particles depends on the choice of individual colours $N_1,..,N_l$ for colouring the link.
\subsection{Topological model for WRT invariants}
The second part of the paper is devoted to the construction of a topological model for the Witten-Reshetikhin-Turaev $3$-manifold invariants. 
Let us fix a level $\cN\in \N$ and let us consider the $2\cN^{th}$ root of unity $\xi=e^{\frac{2 \pi i}{2\cN}}$. We will use the description of closed oriented $3$-manifolds as surgeries along framed oriented links. In turn, we will look at links as closures of braids. Suppose that the corresponding link has $l$ components and the braid has $n$ strands.

We start with the construction of the homology classes in this context. This time we use a covering of the configuration space of $n(\cN-2)+l+1$ particles in the $(2n+3l+1)-$punctured disk and a $\mathbb Z^{2n+1}\oplus \mathbb Z^{3l}\oplus \Z$ local system constructed as above, with $k=n$ and $\bar{l}=l$. We consider the homology groups:
$$H^{-n}_{2n+1,n(\cN-2)+l+1,l} \text { and } H^{-n,\partial}_{2n+1,n(\cN-2)+l+1,l} \text{ which are } \Z[x_1^{\pm 1},...,x_l^{\pm 1},y_1^{\pm 1}..., y_l^{\pm 1}, d^{\pm 1}] \text{-modules.}$$
\vspace{-3mm}
\begin{defn} (Homology classes) Let us fix a set of indices $i_1,...,i_{n} \in \{0,...,\cN-2\}$ and denote by $\bar{i}:=(i_1,...,i_n)$.  We consider the classes given by the geometric supports from the picture below:
$$ {\color{red} \mathscr F_{\bar{i}}^{\cN} \in H^{-n}_{2n+1,n(\cN-2)+l+1,l}} \ (\text{definition } \ref{D:C1}) \ \ \ \ \ \ \ \ \ \ \ \ \ \ \ {\color{dgreen} \mathscr L_{\bar{i}}^{\cN} \in H^{-n,\partial}_{2n+1,n(\cN-2)+l+1,l}} \ ( \text{definition }  \ref{D:C2})$$
\vspace{-8mm}
\begin{figure}[H]
\centering
\includegraphics[scale=0.25]{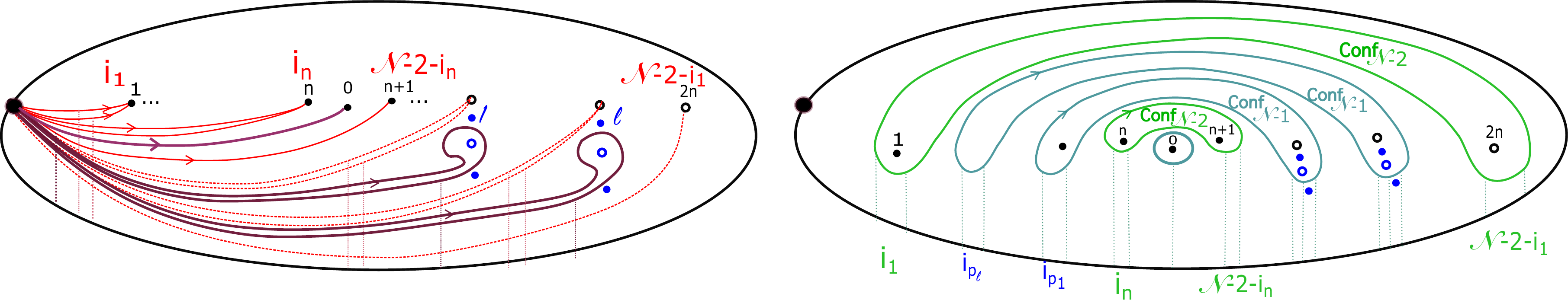}
\caption{WRT Homology Classes}
\label{Localsyst}
\end{figure}
\end{defn}
Denote by $p_1,...,p_l$ a sequence of strands of the braid that correspond to different components of the link and denote by $f_{p_i}$ the framing of the component associated to $p_i$.
\begin{defn}(Lagrangian intersection in the configuration space)\\
For a multi-index $i_1,...,i_{n} \in  \{0,...,\cN-2\}$, we consider the following Lagrangian intersection:
\begin{equation}
\begin{cases}
& \Lambda_{\bar{i}}(\beta_n) \in \Z[x_1^{\pm 1},...,x_l^{\pm 1},y_1^{\pm 1}..., y_l^{\pm 1}, d^{\pm 1}]\\
& \Lambda_{\bar{i}}(\beta_n):=\prod_{i=1}^l x_{C(p_i)}^{ \left(f_{p_i}-\sum_{j \neq {p_i}} lk_{p_i,j} \right)} \cdot \prod_{i=1}^n x^{-1}_{C(i)} 
  \   \left\langle(\beta_{n} \cup {\mathbb I}_{n+3l+1} ) \ { \color{red} \mathscr F_{\bar{i}}^{\cN}}, {\color{dgreen} \mathscr L_{\bar{i}}^{\cN}}\right\rangle \end{cases}
\end{equation} 
\end{defn}
The main result shows that the $\cN^{th} $WRT invariant $\tau_{\cN}(M)$ comes from a state sum of specialisations of these intersections, which take place in the configuration space $Conf_{n(\cN-2)+l+1}(\mathscr D_{2n+3l+1})$. 
\begin{thm}\label{THEOREMW}(Topological state sum model for the Witten-Reshetikhin-Turaev invariants)\\
Let $\cN \in \N$ be a fixed level and $M$ a closed oriented $3$-manifold.   We consider $L$ a framed oriented link with $l$ components such that $M$ is obtained by surgery along $L$. Also, let $\beta_n \in B_n$ such that $L=\widehat{\beta_n}$ as above. Then the $\cN^{th}$ Witten-Reshetikhin-Turaev invariant has the following model:
\begin{equation}
\begin{aligned}
\tau  _{\cN}(M)=\frac{\{1\}^{-l}_{\xi}}{\cD^b \cdot \Delta_+^{b_+}\cdot \Delta_-^{b_-}}\cdot {\Huge{\sum_{i_1,..,i_n=0}^{\cN-2}}}   \left(\sum_{\substack{ \tiny 1 \leq N_1,...,N_l \leq \cN-1 \\ i_k\leq C_k-1}}  \Lambda_{\bar{i}}(\beta_n) \Bigm| _{\psi^{C}_{\xi,N_1,...,N_l}}\right).
\end{aligned}
\end{equation} 
\end{thm}
In this expression $\psi^{C}_{\xi,N_1,...,N_l}$ is a specialisation of variables to complex numbers (see relation \eqref{not}). The coefficients in the above formula are presented in notation \ref{coefffrac}.
\begin{rmk} (Intersections in various configuration spaces) For a fixed colour $\cN $, the algebraic definition of the WRT invariant $\tau_\cN(M) $ is given by a certain linear combination of $J_{N_1,...,N_l}(L,\xi)$ for all $N_1,...,N_l\in \{1,...,\cN-1\}$. Then, Theorem \ref{THEOREMM} would interpret this invariant as follows:
\begin{equation*}
\begin{aligned}
\tau_\cN(M) & \text{ is a linear combination over all } N_1,...,N_l\in \{1,...,\cN-1\}\\
& \text{ and all } \bar{i}=(i_1,...,i_{n}) \text{ with } 0\leq i_k \leq C_k-1, k \in \{1,...,n\}   \ \text{of}\\
& \left\langle(\beta_{n} \cup {\mathbb I}_{n+1} ) \ { \color{red} \mathscr F_{\bar{i}}^{\CC}}, {\color{dgreen} \mathscr L_{\bar{i}}^{\CC}}\right\rangle \Bigm| _{\psi^{C}_{\xi,N_1,...,N_l}}.
\end{aligned}
\end{equation*}
Each term above is an intersection in the configuration space of $1+\sum_{i=1}^{n} (C_i-1)$ particles in the $(2n+1)$-punctured disk, which depends on the choice of colours $N_1,...,N_l$. 

This means that the translation of the algebraic definition of the WRT invariant following Theorem \ref{THEOREMM} shows that $\tau_\cN(M)$ is a linear combination of Lagrangian intersections in different configuration spaces $C_{2n+1,k}$, where the number of particles $k$ varies between $1$ and $(n-1)(\cN-1)+1$.
\end{rmk}
\begin{rmk} (Intersection in a fixed configuration space)
A feature of the model presented in Theorem \ref{THEOREMW} is that it globalises all these intersections from above, showing that the $\cN^{th}$ WRT invariant is given by states of certain Lagrangian intersections in a fixed ambient space.
\begin{equation*}
\begin{aligned}
\tau_\cN(M) & \text{ is a scalar times the state sum over all multi-indices } i_1,...,i_n\in \{0,...,\cN-2\} \text{ of }\\
& \text{specialisations of the intersection } \Lambda_{\bar{i}}(\beta_n) \  \text{corresponding to } N_1,...,N_l\in \{1,...,\cN-1\} \\
& \text{ such that }i_k \leq C_k-1, k \in \{1,...,n\}, \text{ namely: } \Lambda_{\bar{i}}(\beta_n) \Bigm| _{\psi^{C}_{\xi,N_1,...,N_l}}.
\end{aligned}
\end{equation*}
All the intersections $\Lambda_{\bar{i}}(\beta_n)$ above are constructed from the classes $(\beta_{n} \cup {\mathbb I}_{n+3l+1} ) \ \mathscr F_{\bar{i}}^{\cN}$ and $\mathscr L_{\bar{i}}^{\cN}$ and take place in the fixed configuration space of $n(\cN-2)+l+1$ points in the $(2n+3l+1)$-punctured disk. 
\end{rmk}
\begin{rmk} (Encoding the Kirby colour) Now, we discuss the coefficients which appear in the algebraic definition of the Witten-Reshetikhin-Turaev invariant for a $3$-manifold. This formula is given as linear combinations of coloured Jones polynomials of the underlying link $L$ and the coefficients come from the so-called Kirby colour and they are quantum integers. 

Theorem \ref{THEOREMW} provides a globalised formula for $\tau_\cN(M) $ and does not require individual coloured Jones polynomials. For each multi-index $\bar{i}$ bounded by the level, we consider the intersection form $\Lambda_{\bar{i}}(\beta_n)$ between globalised classes in a covering of the configuration space, which does not depend on any colouring. Then, we have to add up its specialisations, corresponding to colours which are ``bigger'' than the index $\bar{i}$. 

The coefficients coming from the Kirby colour are encoded in the homology classes. Geometrically, they are given precisely by the special $l$ purple circles and $l$ blue circles from the supports of the homology classes and they correspond to the orange intersection points from figure \ref{IntersectionForm}.
\end{rmk}
\clearpage
\subsection{Structure of the WRT invariants} Following this remark together with the fact that the intersection pairing is encoded by graded geometric intersections in the base configuration space (remark \ref{pairing}), we conclude that we have a topological formula for the WRT invariant which is obtained from the intersection points between the following geometric supports:
$$(\beta_{n} \cup {\mathbb I}_{n+3l+1} ) \ { \color{red} \mathscr F_{\bar{i}}^{\cN}} \cap  {\color{dgreen} \mathscr L_{\bar{i}}^{\cN}}$$
for all choices of indices $i_1,...,i_n\in \{1,...,\cN-2\}$, graded using the local system $\Phi$. 
\begin{rmk}
 In this way, we see that the WRT invariant at level $\cN$ is completely encoded by the set of intersection points between certain Lagrangian submanifolds in the configuration space of $n(\cN-2)+l+1$ points in the $(2n+3l+1)$-punctured disk.The number of particles is fixed and it is determined by the level of the invariant $\cN$, the number of components of the link $l$ and number of strands of the braid $n$.
\end{rmk}
\begin{figure}[H]
\centering
\includegraphics[scale=0.38]{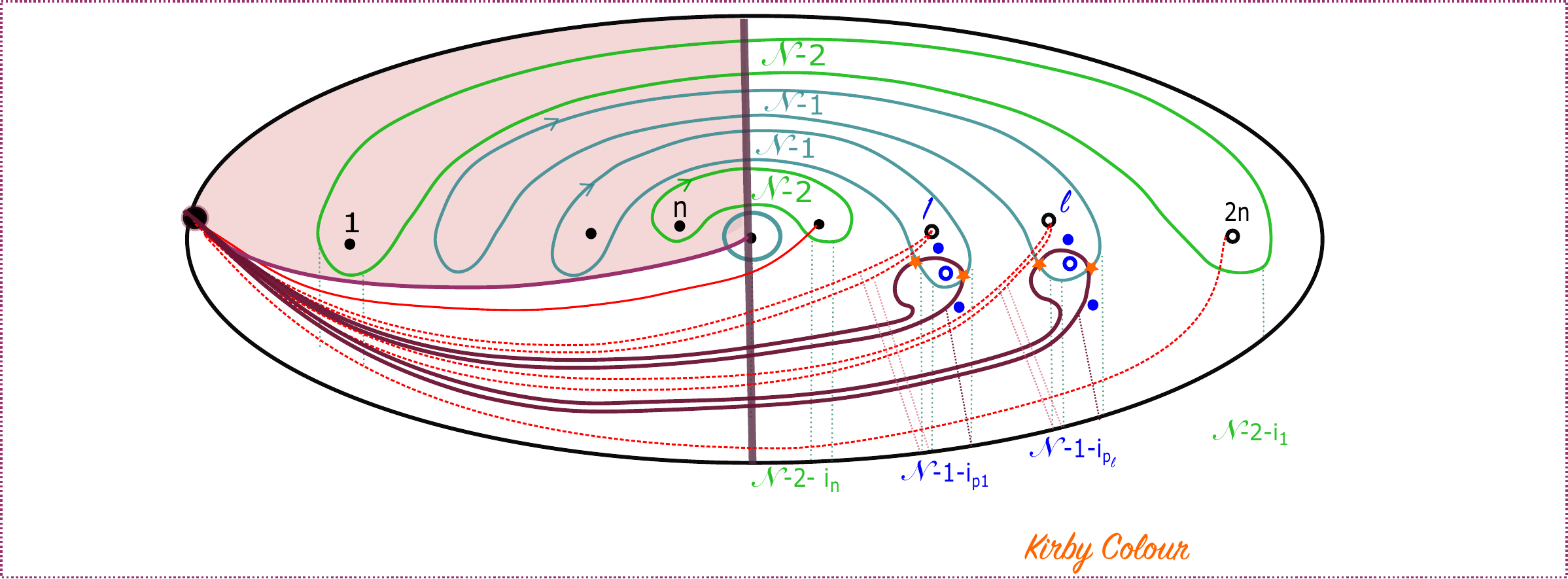}
\caption{Lagrangian Intersection encoding the Kirby colour}
\label{IntersectionForm}
\end{figure}
\vspace{-5mm}
\subsection{Questions-underlying topological information} Our main motivation for this work is the understanding of the underlying topology which is carried by the Witten-Reshetikhin-Turaev invariants. The structural description presented above, provided by intersections between Lagrangians in a fixed configuration space, brings a new approach to investigating further questions concerning categorifications for these quantum invariants. 

\subsection*{Structure of the paper} This article has three main parts. In Section \ref{S:3} we introduce the homological setting that we use as well as the particular choice of a local system and the corresponding covering space and homology groups. In the second part, we construct certain homology classes and, using those, we prove the topological intersection formula for the coloured Jones polynomials for links. The following section \ref{S:5} has two main parts. First, we construct a sequence of homology classes  in a fixed covering space and use them to define a state sum formula. Then, we prove that it leads to a topological model for the $\cN^{th}$ Witten-Reshetikhin-Turaev invariant. In the last part, Section \ref{S:6}, we present the formula for these invariants in the particular case where we have $3$-manifolds which are given by surgeries along knots. 

\subsection*{Acknowledgements} 
 This paper was prepared at the University of Oxford, and I acknowledge the support of the European Research Council (ERC) under the European Union's Horizon 2020 research and innovation programme (grant agreement No 674978). I would like to thank Jacob Rasmussen for discussions and comments on earlier versions of this paper. 
\clearpage
\section{Notations}
In the next sections we will change the variables from the ring of Laurent polynomials using certain specialisations of coefficients. For this, we use the following definition.
\begin{notation}(Specialisation)\label{N:spec}\\ 
Let $N$ be a module over a ring $R$. Let $R'$ be another ring and suppose that we have a specialisation of the coefficients, meaning a morphism:
$$\psi: R \rightarrow R'.$$
We denote by  
$$N|_{\psi}:=N \otimes_{R} R'$$
the specialisation of the module $N$ by the function $\psi$. 
\end{notation}

\begin{defn}(Quantum numbers)
$$ \{ x \}_q :=q^x-q^{-x} \ \ \ \ [x]_{q}:= \frac{q^x-q^{-x}}{q-q^{-1}}.$$
\end{defn}
\begin{defn}(Specialisations of coefficients)\\
For a set of $l$ colours $N_1,..,N_l\in\N$ and a colouring $C:\{1,...,n\}\rightarrow \{1,...,l\}$ we consider the specialisation of coefficients as below:
\begin{center}
\begin{tikzpicture}
[x=1.2mm,y=1.4mm]
\node (b1)               at (-27,0)    {$\Z[x_1^{\pm 1},...,x_n^{\pm 1},y_1^{\pm 1}..., y_{\bar{l}}^{\pm 1}, d^{\pm 1}]$};
\node (b2)   at (27,0)   {$\Z[x_1^{\pm 1},...,x_l^{\pm 1},y_1^{\pm 1}..., y_{\bar{l}}^{\pm 1}, d^{\pm 1}]$};
\node (b3)   at (0,-20)   {$\mathbb Z[q^{\pm1}]$};
,\draw[->]   (b1)      to node[xshift=1mm,yshift=5mm,font=\large]{$f_C$ \eqref{eq:8}}                           (b2);
\draw[->]             (b2)      to node[right,xshift=2mm,font=\large]{$\psi^{C}_{q,N_1,...,N_l}$\ \eqref{eq:8'}}   (b3);
\draw[->,thick,dotted]             (b1)      to node[left,font=\large]{}   (b3);
\end{tikzpicture}
\end{center}
\end{defn}
\begin{defn}(Our setting: specialisation corresponding to a braid closure)\\
We will use this change of coefficients in the situation where $n$ is replaced by $2n+1$ and the $2n$ points except the middle one inherit a colouring with $l$ colours coming from a braid closure of a braid with $n$ strands:
$$C:\{1,...,2n\}\rightarrow \{1,...,l\}.$$
Further on, we consider an extra point in the middle which we denote by $(n+1)$ and we colour it with the label: $$C(n+1)=1.$$
This together with the colouring of the $2n$ points from above gives us a colouring of $2n+1$ points:
\begin{equation}\label{eq:col}
C:\{1,...,2n+1\}\rightarrow \{1,...,l\}.
\end{equation}
For our model, we will use the function $f_C$ corresponding to the colouring from \eqref{eq:col}. Further on, we define the specialisation of coefficients:
$$ \psi^{C}_{q,N_1,...,N_l}: \Z[x_1^{\pm 1},...,x_{l}^{\pm 1},y_1^{\pm 1},...,y_{l}^{\pm 1}, d^{\pm 1}] \rightarrow \Z[q^{\pm 1}]$$
\begin{equation}\label{not}
\begin{cases}
&\psi^{C}_{q,N_1,...,N_l}(x_i)=q^{N_i-1}, \ i\in \{1,...,l\}\\
&\psi^{C}_{q,N_1,...,N_l}(y_i)=q^{N_i}\\
&\psi^{C}_{q,N_1,...,N_l}(d)=q^{-2}.
\end{cases}
\end{equation} 
\end{defn}
\begin{rmk}\label{not'}
In the formulas from the paper, we denote by $C_i:=N_{C(i)}$.
\end{rmk}
\section{Definition of the local system and homology groups}\label{S:3}

In order to construct the classes that will lead to the $3$-manifold invariants, we will use the homology of certain coverings of the configuration space in the punctured disk. The construction of the covering space will be more subtle than the one used in \cite{Cr5}. More specifically, we will consider two types of punctures and use a subtle local system which counts the monodromies around these punctures in different manners. 

For the following part, let us fix $\bar{l},k \in \N$. Also, we consider a ``weight'' $m \in \N$.
 We start with the unordered configuration space of $m$ points in the punctured disk with
  $n+3\bar{l}$ punctures $\mathscr D_{n+3\bar{l}}$, denoted by:
 $$C_{n+3\bar{l},m}.$$ 
Also, we fix $d_1,..d_m \in \partial \hspace{0.5mm}\mathscr D_{n+3\bar{l}}$ and let ${\bf d}=(d_1,...,d_m)$ to be our base point in the configuration space.
Now, we define a certain local system on this configuration space. 
For this, we use the homology of this configuration space, which has the following description.
\begin{prop}
Let us suppose that $m \geq 2$. Let $[ \ ]: \pi_1(C_{n+3\bar{l},m}) \rightarrow H_1\left( C_{n+3\bar{l},m}\right)$ be the abelianisation map. Then the homology has the following form:
\begin{equation*}
\begin{aligned}
H_1\left( C_{n+3\bar{l},m}\right)  \simeq \ \ \ \  \ & \Z^{n} \ \ \ \ \oplus \ \ \ \  \Z^{2\bar{l}} \ \ \ \ \ \oplus \ \ \ \ \Z^{\bar{l}} \ \ \ \  \oplus \ \ \ \ \Z\\
&\langle [\sigma_i] \rangle \ \ \ \ \  \langle [\gamma_j], [\bar{\gamma}_j] \rangle \ \ \ \ \ \ \ \langle [\eta_j] \rangle \ \ \ \ \ \ \ \langle [\delta]\rangle,  \ \ \ {i\in \{1,...,n\}}, j\in \{1,...,\bar{l}\}.
\end{aligned}
\end{equation*}
The five types of generators are presented in the picture below. 
\begin{figure}[H]
\centering
\includegraphics[scale=0.23]{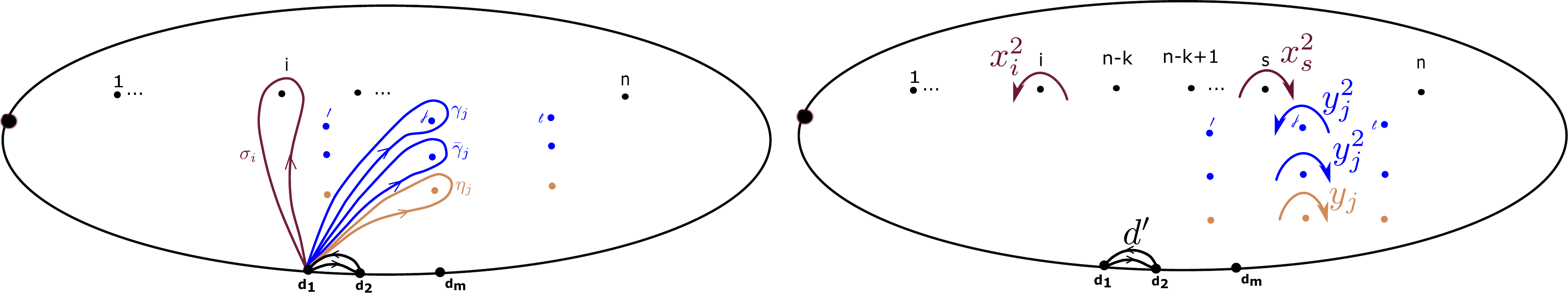}
\caption{Local system $\Phi$}
\label{Localsyst}
\end{figure}
\end{prop}
\vspace{-5mm}
We continue with the augmentation map $$\epsilon: H_1\left( C_{n+3l,m}\right)\rightarrow \Z^n \oplus \Z^l \oplus \Z$$ 
$$ \hspace{30mm} \langle x_i \rangle \ \ \langle y_j\rangle \ \ \langle d' \rangle$$ given by:
\begin{equation}
\begin{cases}
&\epsilon(\sigma_i)=2x_i, i\in \{1,...,n-k\}\\
&\epsilon(\sigma_i)=-2x_i, i\in \{n-k+1,...,n\}\\
&\epsilon(\gamma_j)=2y_j, j\in \{1,...,\bar{l}\}\\
&\epsilon(\bar{\gamma}_j)=-2y_j, j\in \{1,...,\bar{l}\}\\
&\epsilon(\eta_j)=-y_j, j\in \{1,...,\bar{l}\}\\
&\epsilon(\delta)=d'.
\end{cases}
\end{equation}
\begin{defn}(Local system)\label{localsystem}
We use the local system given by the composition of the above morphisms:
\begin{equation}
\begin{aligned}
&\Phi: \pi_1(C_{n+3\bar{l},m}) \rightarrow \Z^n \oplus \Z^{\bar{l}} \oplus \Z\\
&\hspace{28mm} \langle x_j \rangle \ \ \langle y_j\rangle \ \ \langle d' \rangle, \ i \in \{1,...,n\}, \ j \in \{1,...,\bar{l}\}\\
&\Phi= \epsilon \circ [ \ ]. \ \ \ \ \ \ \ \ \ \ \ \ \ \ \ \ \ \ \ \ 
\end{aligned}
\end{equation}
\end{defn}

\begin{defn}(Covering of the configuration space)\label{D:12}\\
Let $\tilde{C}^{-k}_{n+3{\bar{l}},m}$ be the covering of $C_{n+3{\bar{l}},m}$ corresponding to the local system $\Phi$. 

Also, let us fix a base point ${\bf\tilde{d}}\in \tilde{C}^{-k}_{n+3{\bar{l}},m}$ in the fiber over the base point $\bf{d}$.
\end{defn}
\subsection{Input of the construction}
We will use the homologies of this covering space. They are modules over the group ring of deck transformations, $\Z[x_1^{\pm 1},...,x_n^{\pm 1},y_1^{\pm 1}..., y_{\bar{l}}^{\pm 1}, d'^{\pm 1}]$.

For computational purposes, we will use the variable $d:=-d'$ and we consider:
\begin{equation}
\begin{aligned}
&\gamma: \Z[x_1^{\pm 1},...,x_n^{\pm 1},y_1^{\pm 1}..., y_{\bar{l}}^{\pm 1}, d'^{\pm 1}] \rightarrow \Z[x_1^{\pm 1},...,x_n^{\pm 1},y_1^{\pm 1}..., y_{\bar{l}}^{\pm 1}, d^{\pm 1}]\\
&\begin{cases}
\bar{\Phi}(x_i)=x_i\\
\bar{\Phi}(y_j)=y_j\\
\bar{\Phi}(d')=-d.
\end{cases}
\end{aligned}
\end{equation}
Using this notation, the homology groups of the covering become modules over\\
 $\Z[x_1^{\pm 1},...,x_n^{\pm 1},y_1^{\pm 1}..., y_{\bar{l}}^{\pm 1}, d^{\pm 1}]$.
Further on, we use the induced map corresponding to the local system $\Phi$ with values in the group ring of $\Z^n \oplus \Z^{\bar{l}} \oplus \Z$ :
\begin{equation}
\Phi: \pi_1(C_{n+3{\bar{l}},m}) \rightarrow \Z[x_1^{\pm 1},...,x_n^{\pm 1},y_1^{\pm 1}..., y_{\bar{l}}^{\pm 1}, d^{\pm 1}].
\end{equation}
Then, taking into account the change of variables $\gamma$, we define:
\begin{equation}
\begin{aligned}
&\bar{\Phi}: \pi_1(C_{n+3{\bar{l}},m}) \rightarrow \Z[x_1^{\pm 1},...,x_n^{\pm 1},y_1^{\pm 1}..., y_{\bar{l}}^{\pm 1}, d^{\pm 1}]\\
&\bar{\Phi}=\gamma \circ \Phi. \ \ \ \ \ \ \ \ \ \ \ \ \ \ \ \ \ \ \ \ 
\end{aligned}
\end{equation}
\begin{defn}
We consider two submodules in the homologies of this covering space (which are modules over $\Z[x_1^{\pm 1},...,x_n^{\pm 1},y_1^{\pm 1}..., y_{\bar{l}}^{\pm 1}, d^{\pm 1}]$):
\begin{enumerate}
 \item[$\bullet$]  $\mathscr H^{-k}_{n,m,{\bar{l}}}\subseteq H^{\text{lf},\infty,-}_m(\tilde{C}^{-k}_{n+3{\bar{l}},m}, P^{-1};\Z)$ and 
 \item[$\bullet$]  $\mathscr H^{-k,\partial}_{n,m,{\bar{l}}} \subseteq H^{\text{lf},\Delta}_m(\tilde{C}^{-k}_{n+3{\bar{l}},m},\partial;\Z)$ 
\end{enumerate}
given by the images of the homology with twisted coefficients into the homology of the covering space, defined in an analogue manner as the homology groups from \cite{Cr5}-Section 3 (using the splitting of the boundary of the configuration space and its description from \cite{CrM}). 

\end{defn}
\begin{prop}(\cite{CrM})\label{P:3'''}
There exists a topological intersection pairing:
$$<< ~,~ >>: \mathscr H^{-k}_{n,m,{\bar{l}}} \otimes \mathscr H^{-k,\partial}_{n,m,{\bar{l}}}\rightarrow\Z[x_1^{\pm 1},...,x_n^{\pm 1},y_1^{\pm 1}..., y_{\bar{l}}^{\pm 1}, d^{\pm 1}].$$
\end{prop}
In the next section, we will use the exact form of this intersection pairing, so we will briefly explain its formula. 
Let us consider two classes $H_1 \in \mathscr H^{-k}_{n,m,{\bar{l}}}$ and $H_2 \in \mathscr  H^{-k,\partial}_{n,m,{\bar{l}}}$. We suppose that these classes are given by the lifts $\tilde{X}_1, \tilde{X}_2$ of two immersed submanifolds $X_1,X_2 \subseteq C_{n+3{\bar{l}},m}$. Also, we assume that $X_1$ and $X_2$ have a transverse intersection, in a finite number of points. 
\begin{prop}(Intersection pairing from intersections in the base space and the local system)\label{P:3}  For each intersection point $x \in X_1 \cap X_2$ we define a certain loop and denote it by $l_x \subseteq C_{n+3{\bar{l}},m}$. 
a) {\bf Construction of $l_x$}\\
 We suppose that we have the paths $\gamma_{X_1}, \gamma_{X_2}$ which start in $\bf d$, they end on $X_1$,$X_2$ respectively and that  
$\tilde{\gamma}_{X_1}(1) \in \tilde{X}_1$ and $ \tilde{\gamma}_{X_2}(1) \in \tilde{X}_2$.
Further on, we choose two paths $\delta_{X_1}, \delta_{X_2}:[0,1]\rightarrow C_{n+3{\bar{l}},m}$ with the property:
\begin{equation}
\begin{cases}
Im(\delta_{X_1})\subseteq X_1; \delta_{X_1}(0)=\gamma_{X_1}(1);  \delta_{X_1}(1)=x\\
Im(\delta_{X_2})\subseteq X_2; \delta_{X_2}(0)=\gamma_{X_2}(1);  \delta_{x_2}(1)=x.
\end{cases}
\end{equation}
The composition of these paths gives us the loop:
$$l_x=\gamma_{X_1}\circ\delta_{X_1}\circ \delta_{X_2}^{-1}\circ \gamma_{X_2}^{-1}.$$
Also, let $\alpha_x$ be the sign of the geometric intersection between $M_1$ and $M_2$ in the base configuration space, at the point $x$.\\
b) {\bf Intersection form}\\
 Then, the intersection pairing can be computed from the set of loops $l_x$ and the local system:
\begin{equation}\label{eq:1}  
<<H_1,H_2>>=\sum_{x \in X_1 \cap X_2} \alpha_x \cdot \Phi(l_x) \in \Z[x_1^{\pm 1},...,x_n^{\pm 1},y_1^{\pm 1}..., y_{\bar{l}}^{\pm 1}, d^{\pm 1}].
\end{equation}
\end{prop}
\begin{rmk}\label{orientd}
For actual computations, in the case where the homology classes come from product of one dimensional submanifolds quotiented in the configuration space, one can replace the variable $d'$ by $d$ and the local system $\Phi$ by $\bar{\Phi}$ in the previous formula, and then count just the product of local orientations in the disk around each component of the  intersection point $x$ (instead of keeping track of the sign of orientations in the configuration space $\alpha_x$). 
\end{rmk}
\subsection{Specialisations given by colorings}
\begin{defn}(Change of coefficients)
For the next part, we suppose that we have a coloring $C$ of the $n$ punctures of the disk into $l$ colours:
\begin{equation}
C:\{1,...,n\}\rightarrow \{1,...,l\}.
\end{equation} 
We will work in the situation where ${\bar{l}}=0$ or ${\bar{l}}=l$. Then, we fix ${\bar{l}}$ components $\bar{p}_1,...,\bar{p}_{\bar{l}}\in \{1,...,n\}$.

Then, we define the corresponding change of variables, where we change the first $n+{\bar{l}}$ components $x_1,...,x_n,y_1,...,y_{\bar{l}}$ from the ring $\Z[x_1^{\pm 1},...,x_n^{\pm 1},y_1^{\pm 1}..., y_{\bar{l}}^{\pm 1}, d^{\pm 1}]$ to $l+{\bar{l}}$ variables, denoted by $x_1,..,x_l,y_1,...,y_{\bar{l}}$,  as below:
$$ f_C: \Z[x_1^{\pm 1},...,x_n^{\pm 1},y_1^{\pm 1}..., y_{\bar{l}}^{\pm 1}, d^{\pm 1}] \rightarrow \Z[x_1^{\pm 1},...,x_l^{\pm 1},y_1^{\pm 1}..., y_{\bar{l}}^{\pm 1}, d^{\pm 1}]$$
\begin{equation}\label{eq:8} 
\begin{cases}
&f_C(x_i)=x_{C(i)}, \ i\in \{1,...,n\}\\
&f_C(y_j)=y_{C(\bar{p}_j)}, \ j\in \{1,...,{\bar{l}}\}.
\end{cases}
\end{equation}
\end{defn}
Now, we will change the coefficients of the homology groups using the function $f_C$.
\begin{defn}(Homology groups)\label{D:4} Let us define the homologies which correspond to these coefficients, given by:
\begin{enumerate}
 \item[$\bullet$]  $H^{-k}_{n,m,{\bar{l}}}:=\mathscr H^{-k}_{n,m,{\bar{l}}}|_{f_C}$ 
 \item[$\bullet$]  $H^{-k,\partial}_{n,m,{\bar{l}}}:=\mathscr H^{-k,\partial}_{n,m,{\bar{l}}}|_{f_C}.$
\end{enumerate}
They are modules over $\Z[x_1^{\pm 1},...,x_l^{\pm 1},y_1^{\pm 1}..., y_{\bar{l}}^{\pm 1}, d^{\pm 1}]$.
\end{defn}
Now, we look at braids with $n+3{\bar{l}}$ strands which preserve the colouring $C$ and the induced colouring on the components $\bar{p}_1,...,\bar{p}_l$ and denote the set of such braids by $B^{C}_{n+3{\bar{l}}}$.
\begin{prop}(\cite{CrM}) \label{colbr} There is a braid group action (which comes from the mapping class group action) which is compatible with the action of deck transformations at the homological level:
$$B^{C}_{n+3{\bar{l}}} \curvearrowright H^{-k}_{n,m,{\bar{l}}} \ \left( \text{ as a module over } \Z[x_1^{\pm 1},...,x_l^{\pm 1},y_1^{\pm 1}..., y_{\bar{l}}^{\pm 1}, d^{\pm 1}]\right).$$ 
\end{prop}
\begin{prop}(\cite{CrM})\label{P:3'}
There is also a topological intersection pairing:
$$\left\langle ~,~ \right\rangle:  H^{-k}_{n,m,{\bar{l}}} \otimes H^{-k,\partial}_{n,m,{\bar{l}}}\rightarrow\Z[x_1^{\pm 1},...,x_l^{\pm 1},y_1^{\pm 1}..., y_{\bar{l}}^{\pm 1}, d^{\pm 1}].$$
whose method of computation is the same as the one presented in Proposition \ref{P:3}, specialised using the change of coefficients $f_C$:
$$ \left\langle ~,~ \right\rangle= \  << ~,~>>|_{f_C}.$$
\end{prop}
\begin{defn}(Specialisation of coefficients)
Let $N_1,...,N_l\in \N$ a sequence of natural numbers. We define the specialisation of coefficients given by:
$$ \psi^{C}_{q,N_1,...,N_l}: \Z[x_1^{\pm 1},...,x_l^{\pm 1},y_1^{\pm 1}..., y_l^{\pm 1}, d^{\pm 1}] \rightarrow \Z[q^{\pm 1}]$$
\begin{equation}\label{eq:8'} 
\begin{cases}
&\psi^C_{q,N_1,...,N_l}(x_i)=q^{N_i-1}, \ i\in \{1,...,l\}\\
&\psi^C_{q,N_1,...,N_l}(y_i)=q^{N_i}, \ i\in \{1,...,{\bar{l}}\}\\
&\psi^C_{q,N_1,...,N_l}(d)=q^{-2}.
\end{cases}
\end{equation}
\end{defn}
\section{Coloured Jones polynomials for framed links}\label{S:4}
In this section, we show a topological intersection formula for coloured Jones polynomials for links whose components are coloured with different colours. 

Let us start with $L=K_1 \cup ...\cup K_l$ a framed oriented link with framings $f_1,...,f_l\in \Z$. 
 Let us choose $\beta_n \in B_n$ a braid such that  $L= \widehat{\beta_n}$. We also fix a set of colours $N_1,...,N_l\in \N$.

\begin{notation}
For a natural number $M\in \N$, we denote by $V_{M}$ the $M$-dimensional representation of the quantum group $U_q(sl(2))$.

We colour the components of the link $L$ with the representations $V_{N_1},...,V_{N_l}$ and denote the coloured Jones polynomial of this framed link by $J_{N_1,...,N_l}(L,q)$ (as in \cite{Turaev}).
Also, for the further notations, we consider: $$\bar{N}:=(N_1,...,N_l).$$
\end{notation}
\begin{defn}(Induced colorings)\\
a) (Colourings of the braid)   The colouring of the link given by $\bar{N}$ induces a colouring of the strands of the braid, and we denote the corresponding colours by:
$$(C_1,...,C_n).$$ 
Now, we look at the link as the closure of the braid $\beta_n$ together with $n$ straight strands, and so, we have an associated colouring of $2n$ points $C:\{1,...,2n\}\rightarrow \{1,...,l\}$. 
This means that we have the following colours on the $2n$ points:
$$(C_1,...,C_n,C_n,...,C_1).$$ 
We will work with the $(2n+1)$-punctured disk and for this purpose we define a colouring of $2n+1$ points as below:
$$\bar{C}^{\bar{N}}:=(C_1,...,C_n,N_1,C_n,...,C_1).$$
 b) (Set of states) We consider the following indexing set:
 
 $$C(\bar{N}):= \big\{ \bar{i}=(i_1,...,i_n)\in \N^{n} \mid 0\leq i_k \leq C_k-1, \  \forall k\in \{1,...,n\} \big\}.$$ 
\end{defn}
\subsection{Homology classes}
Now that we have the induced colouring of the braid and the corresponding indexing set $C(\bar{N})$, we can present the homology groups that we will use. More specifically, we will use the configuration space of $1+\sum_{i=1}^{n} (C_i-1)$ points on the $(2n+1)$-punctured disk. Then, we consider the  covering coming from the local system $\Phi$ associated to the parameters:
$$ n \rightarrow 2n+1; \ \ \ m\rightarrow 1+\sum_{i=1}^{n} (C_i-1); \ \ \ {\bar{l}}\rightarrow 0; \ \ \ k\rightarrow -n.$$ 
We use the corresponding homology groups:
$$H^{-n}_{2n+1, 1+\sum_{i=1}^{n} (C_i-1),0} \ \ \ \ \ \ \ \ \ \ \ \ \ \ \ \text{ and }\ \ \ \ \ \ \ H^{-n,\partial}_{2n+1,1+\sum_{i=1}^{n} (C_i-1),0}.$$
For the following part, since the third component is zero, we will just erase it from the indices of the homology groups.
Now we are ready to define the homology classes that will be used in the intersection model. The classes will be prescribed by a couple given by:
\begin{itemize}
\item[•] A {\em geometric support}, meaning a {\em set of arcs in the punctured disk}. The image of the product of these arcs in the configuration space, gives us a submanifold which has half of the dimension of the configuration space. 
\item[•] A set of {\em paths to the base point}, which start in the base points from the punctured disk and end on these curves. The set of these paths gives a path in the configuration space, from $\bf d$ to the submanifold mentioned above. 
\end{itemize}
Then, we lift the path to a path in the covering space, starting from $\tilde{\bf{d}}$ and then we lift the submanifold through the end point of this path. The detailed construction of such homology classes is presented in \cite{Cr5}, Section 5. 
\begin{defn} (Homology classes)\\
 For any set of indices $\bar{i}=(i_1,...,i_{n}) \in C(\bar{N})$  we define two homology classes, given by the geometric supports from figure \ref{Picture0}:
 
$${\color{red} \mathscr F_{\bar{i}}^{\CC} \in H^{-n}_{2n+1, {\scriptscriptstyle 1+\sum_{i=1}^{n}}(C_i-1)}} \ \ \ \ \ \ \ \ \ \ \ \ \ \ \ \text{ and }\ \ \ \ \ \ \ \ \ \ \ \ \ \ \ \ \ \  {\color{dgreen} \mathscr L^{\CC}_{\bar{i}}\in H^{-n,\partial}_{2n+1,1+\sum_{i=1}^{n} (C_i-1)}}.$$
$$\hspace{5mm}\downarrow \text{ lifts }$$
\vspace{-13mm}
\begin{figure}[H]
\centering
\includegraphics[scale=0.4]{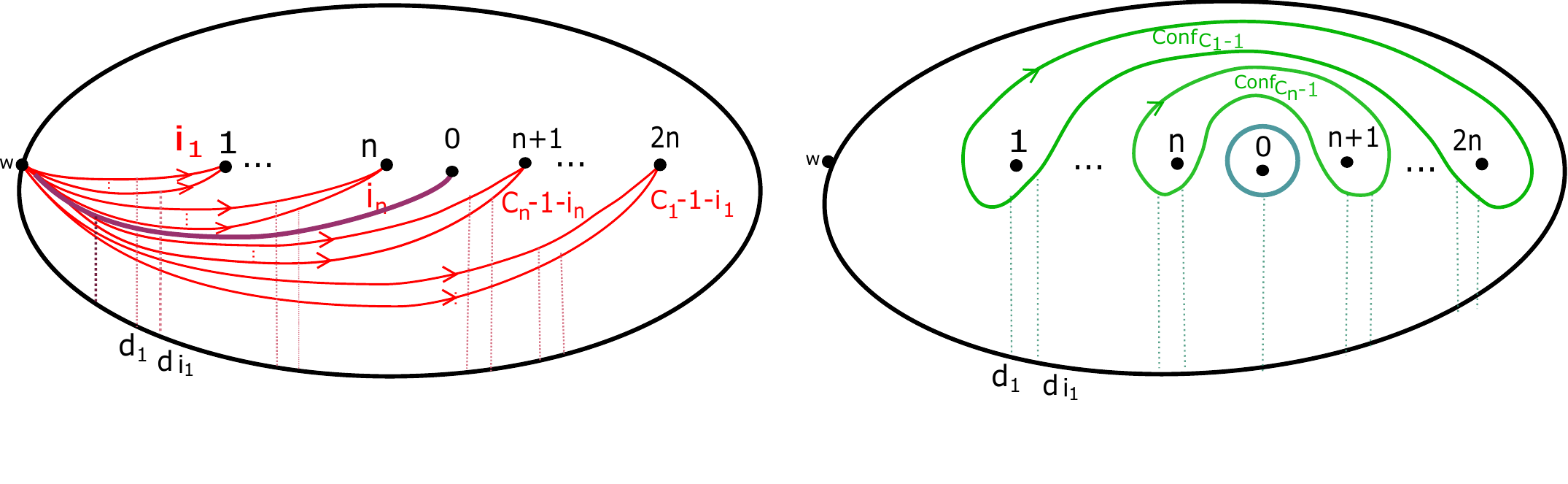}
\vspace{-5mm}
\caption{Embedded Lagrangians}
\label{Picture0}
\end{figure}
\end{defn}
In the next part, we use the specialisation of coefficients:
$$ \psi^{C}_{q,N_1,...,N_l}: \Z[x_1^{\pm 1},...,x_{l}^{\pm 1}, d^{\pm 1}] \rightarrow \Z[q^{\pm 1}]$$
\begin{equation}\label{eq:8''''} 
\begin{cases}
&\psi^{C}_{q,N_1,...,N_l}(x_i)=q^{N_i-1}, \ i\in \{1,...,l\}\\
&\psi^{C}_{q,N_1,...,N_l}(d)=q^{-2}.
\end{cases}
\end{equation}
\subsection{Intersection model}
Now we show that the coloured Jones polynomial of a link coloured with the colours $N_1,...,N_l$ can be obtained from an intersection pairing which uses the classes $\mathscr F_{\bar{i}}^{\CC}$ and $\mathscr L_{\bar{i}}^{\CC}$ for all $\bar{i} \in C(\bar{N})$.
\begin{thm}\label{THEOREM}(Topological state sum model for coloured Jones polynomials for coloured links)

\begin{equation}
\begin{aligned}
J_{N_1,...,N_l}(L,q)& =~ q^{ \sum_{i=1}^{l}\left( f_i- \sum_{j \neq i} lk_{i,j}\right)(N_i-1)} \cdot \\
 & \cdot \left(\sum_{\bar{i}\in C(\bar{N})} \left( \prod_{i=1}^{n}x^{-1}_{C(i)} \right)\cdot  \left\langle(\beta_{n} \cup {\mathbb I}_{n+1} ) \ { \color{red} \mathscr F_{\bar{i}}^{\CC}}, {\color{dgreen} \mathscr L_{\bar{i}}^{\CC}}\right\rangle \right)\Bigm| _{\psi^{C}_{q,N_1,...,N_l}}.
\end{aligned}
\end{equation} 
In this formula we denote by $(lk_{i,j})_{i,j \in \{1,..,l\}}$ the linking matrix of the link $L$.
\end{thm}
\begin{proof}
The proof of this intersection formula is a generalisation of the strategy used in the model for coloured Jones polynomials coloured with the same colour, presented in \cite{Cr5}, based on arguments from \cite{Cr3}. We outline the main steps as follows.\\
{\bf Step 1} We consider the homology classes
$$ \bar{\mathscr F}_{\bar{i}}^{\CC} \in H^{-n}_{2n, {\scriptscriptstyle \sum_{i=1}^{n}}(C_i-1)} \ \ \ \ \ \ \text{ and }\ \ \ \ \ \ \ \ \  { \bar{\mathscr L}^{\CC}_{\bar{i}}\in H^{-n,\partial}_{2n,\sum_{i=1}^{n} (C_i-1)}}$$
which have the same geometric support as the classes $\mathscr F_{\bar{i}}^{\CC}$ and $\mathscr L_{\bar{i}}^{\CC}$ except that we remove the 1-dimensional part which is supported around the puncture labeled by $0$, namely the purple segment and the blue circle (see a similar argument in Step 2, Section 6 from \cite{Cr5}). Then we have that:
$$ \left\langle(\beta_{n} \cup {\mathbb I}_{n+1} ) \ {  \mathscr F_{\bar{i}}^{\CC}}, { \mathscr L_{\bar{i}}^{\CC}}\right\rangle= d^{- \sum_{k=1}^{n}i_k}  \left\langle(\beta_{n} \cup {\mathbb I}_{n} ) \ { \bar{ \mathscr F}_{\bar{i}}^{\CC}}, {\bar{\mathscr L}_{\bar{i}}^{\CC}}\right\rangle $$
This means that we want to prove the following: 
\begin{equation}
\begin{aligned}
J_{N_1,...,N_l}(L,q)& =~ q^{ \sum_{i=1}^{l}\left( f_i- \sum_{j \neq i} lk_{i,j}\right)(N_i-1)} \cdot \\
 & \cdot \left(\sum_{\bar{i}\in C(\bar{N})} \prod_{i=1}^{n}x^{-1}_{C(i)} \cdot d^{- \sum_{k=1}^{n}i_k} \left\langle(\beta_{n} \cup {\mathbb I}_{n} ) \ { \color{red} \bar{\mathscr F}_{\bar{i}}^{\CC}}, {\color{dgreen} \bar{\mathscr L}_{\bar{i}}^{\CC}}\right\rangle \right)\Bigm| _{\psi^{C}_{q,N_1,...,N_l}}.
\end{aligned}
\end{equation} 
{\bf Step 2} For the next part, we follow step by step the correspondence to the Reshetikhin-Turaev definition of the coloured Jones polynomials. More specifically, the cups of the diagram correspond to the sum of the classes $\bar{\mathscr F}_{\bar{i}}^{\CC}$ over all $\bar{i}\in C(\bar{N})$. 
Further on, the braid action on the quantum side and on the homological side correspond, using the identification due to Martel \cite{Martel}. 

{\bf Step 3}
In the end, the caps of the diagram require that after the braid group action, we evaluate just the components which are symmetric with respect to the middle of the disc. More precisely this means that the indices corresponding to the points $k$ and $2n+1-k$ should sum up to the colour $C_k-1$. This is encoded geometrically by the intersection with the dual class $\bar{\mathscr L}_{\bar{i}}^{\CC}$. 

On the algebraic side, one should also encode an extra coefficient which corresponds to the caps of the diagram. We reffer to the details of the argument for a single colour as they are presented in \cite{Cr3} (Section 5 and Section 7), except that here we have a different local system. The fact that the change of the local system does not affect the flow of the proof follows by a similar computation as the one from Step 3, Section 6 from \cite{Cr5}. The main points are as follows. 

This coefficient is given by the pivotal structure, more specifically by the action of the element $K^{-1}$ from the quantum group. Now, for a set of indices $i_1,...,i_n$ the $K^{-1}$ action on the corresponding tensor monomial is given by:
\begin{equation}
 q^{- \sum_{k=1}^{n}\left( (C_k-1)-2i_k\right)}= \left( \prod_{k=1}^{n}q^{- (C_k-1)} \right) \cdot  q^{ \ \sum_{k=1}^{n}2i_k}.
\end{equation}
This coefficient is precisely the specialisation:
\begin{equation}
\psi^{C}_{N_1,...,N_l} \left(\left(\prod_{i=1}^{n}x^{-1}_{C(i)} \right) \cdot d^{- \sum_{k=1}^{n}i_k}\right).
\end{equation}
The remaining coefficient which appears in the formula comes from the framing contribution of the components of the link $L$.
\end{proof}

\section{WRT from intersections in configuration spaces}\label{S:5}
In this part we pass towards invariants for $3$-manifolds and aim to construct the intersection model for the Witten-Reshetikhin-Turaev invariants, as presented in Theorem \ref{THEOREM}.
Let us fix a level $\cN \in \N$. As in the previous section, we start with a framed oriented link with $l$ components, which is the closure of a braid with $n$ strands.
\begin{defn}(Choice of $l$ points)\\
Let us choose $l$ strands of the braid $\beta_n$ which all belong to different components of the link and denote their indices by:
$p_1,...,p_l$. Also, we look in the $2n+1$ punctured disk and denote the symmetric of these points with respect to the middle axis by $\bar{p}_1,...,\bar{p}_l$.
\end{defn}

This time we will use the homology of the covering of the configuration space of $n(\cN-2)+l+1$ particles in the punctured disk with $2n+3l+1$ punctures, associated to the parameters:
$$ n \rightarrow 2n+1; \ \ \ m\rightarrow n(\cN-2)+l+1; \ \ \ l={\bar{l}}; \ \ \ k\rightarrow -n.$$ 
More precisely, we will work with the homology groups:
$$H^{-n}_{2n+1, n(\cN-2)+l+1,l} \ \ \ \ \ \ \text{ and }\ \ \ \ \ \ \ H^{-n,\partial}_{2n+1,n(\cN-2)+l+1,l}.$$
On the picture, we consider $3l$ blue punctures in the punctured disk such that they are split into triples which lie below the privileged punctures $\bar{p}_1,...,\bar{p}_l$, as in figure \ref{Picture}.
Now we are ready to define the main tools in our construction, which are the homology classes in the homology presented above.

\subsection{Homology classes}
\begin{defn} (First Homology class)\label{D:C1}\\
 For a set of indices $i_1,...,i_{n} \in \{0,...,\cN-2\}$ we denote $\bar{i}:=(i_1,...,i_n)$ and we consider the class given by the geometric support from picture \ref{Picture}:
 
$${\color{red} \mathscr F_{\bar{i}}^{\cN} \in H^{-n}_{2n+1,n(\cN-2)+l+1,l}}$$
\vspace{-10mm}
\begin{figure}[H]
\centering
\includegraphics[scale=0.4]{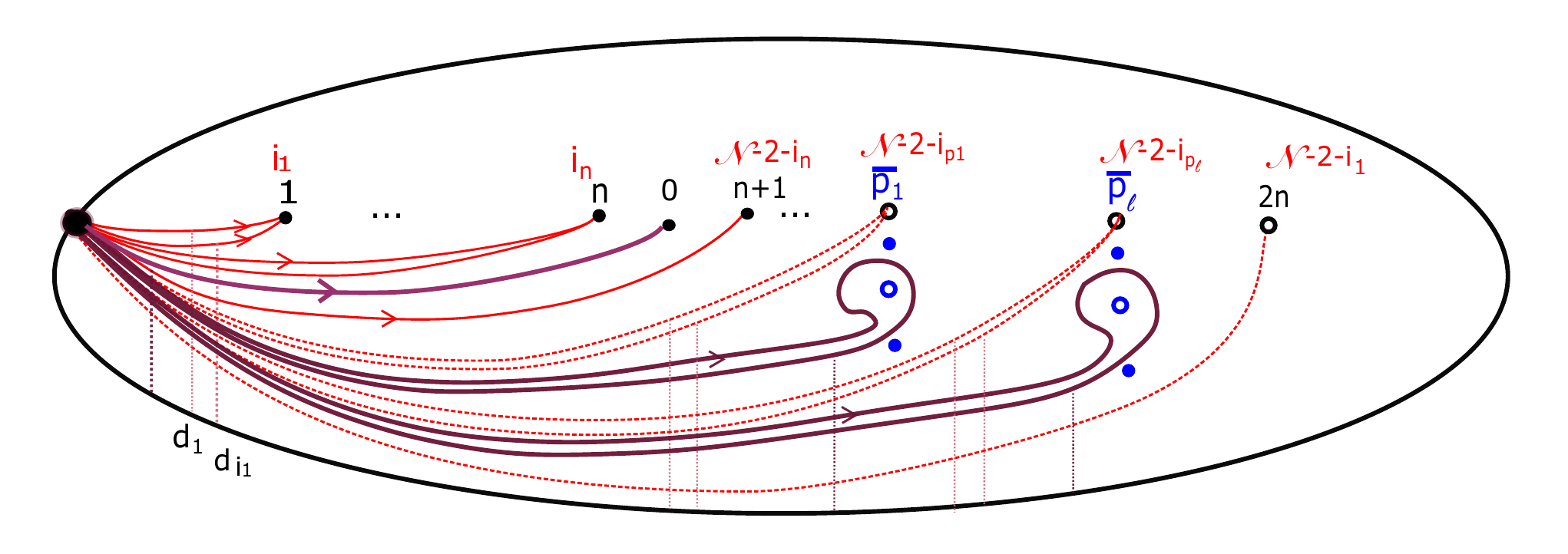}
\caption{}
\label{Picture}
\end{figure}
\end{defn}

\begin{rmk}
When we take one of the circles from the above picture, its lift has a non-trivial monodromy, so this corresponds to an arc which starts and ends in the fiber over $w$. This shows that the lift of the geometric support from figure \ref{Picture} will lead to a well defined homology class in the homology relative to the fiber $P^{-1}$.
\end{rmk}
\begin{defn} (Second Homology class)\label{D:C2}\\
 Also for each choice of indices $i_1,...,i_{n} \in \{0,...,\cN-2\}$ we consider the geometric support given by the product of configuration spaces on the circles from figure \ref{Picture2} and define the associated homology class as below:
 $${\color{dgreen} \mathscr L_{\bar{i}}^{\cN} \in H^{-n,\partial}_{2n+1,n(\cN-2)+l+1,l}}$$
\vspace{-10mm}
\begin{figure}[H]
\centering
\includegraphics[scale=0.4]{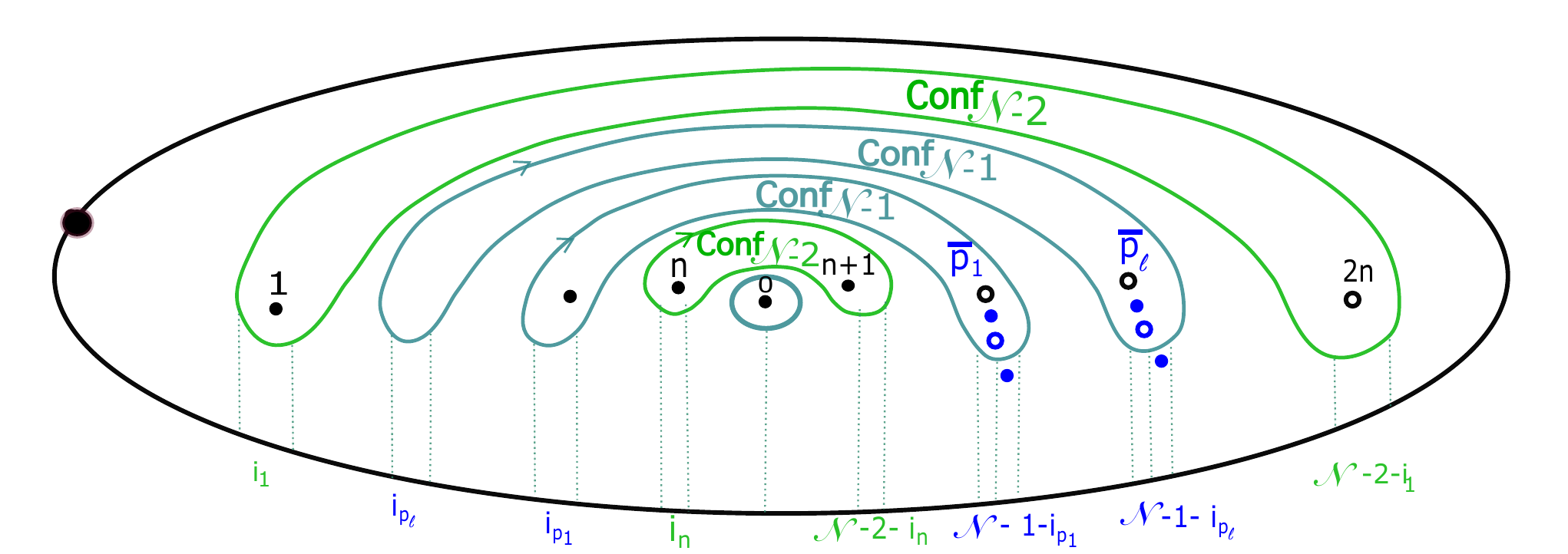}
\caption{}
\label{Picture2}
\end{figure}
\end{defn}
\begin{rmk}
All circles from the above picture have trivial monodromy, since the local system evaluates symmetric points with opposite monodromies and also it evaluates in opposite directions the loops around the blue punctures which are displayed on the vertical directions (and lie in the disks bounded by those circles). So, the geometric support from figure \ref{Picture2} leads to a well defined homology class in $H^{-n,\partial}_{2n+1,n(\cN-2)+l+1,l}$. 
\end{rmk}
We remind the definition of the specialisation of coefficients which is  associated to this context:
$$ \psi^{C}_{q,N_1,...,N_l}: \Z[x_1^{\pm 1},...,x_{l}^{\pm 1},y_1^{\pm 1},...,y_{l}^{\pm 1}, d^{\pm 1}] \rightarrow \Z[q^{\pm 1}]$$
\begin{equation}\label{not'}
\begin{cases}
&\psi^{C}_{q,N_1,...,N_l}(x_i)=q^{N_i-1}, \ i\in \{1,...,l\}\\
&\psi^{C}_{q,N_1,...,N_l}(y_i)=q^{N_i}\\
&\psi^{C}_{q,N_1,...,N_l}(d)=q^{-2}.
\end{cases}
\end{equation} 

\subsection{WRT from intersections in configuration spaces}

\

\begin{defn}(Kirby colour)\\
For $\cN\in \N$, the Kirby colour corresponding to the quantum group $U_{\xi}(sl(2))$ (\cite{Turaev}) is given by:
\begin{equation}
\Omega:=\sum_{N=1}^{\cN-1}qdim(V_N) \cdot V_N= \sum_{N=1}^{\cN-1}[N]_{\xi} \cdot V_N.
\end{equation}
\end{defn}
\begin{notation}
We denote by $b_{+},b_{-}$ and $b$ the number of positive, negative and zero eigenvalues of the linking matrix of $L$. Also, we consider:
\begin{equation}\label{coefffrac}
\begin{aligned}
&\Delta_+=J_{\Omega}(\cU_{+},\xi)\\
&\Delta_-=J_{\Omega}(\cU_{-},\xi)\\
&\cD= \ \mid \Delta_+  \mid 
\end{aligned}
\end{equation}
where $\cU_{+}$ and $\cU_{-}$ are the unknot with framing $+1$ and $-1$ respectively (\cite{Turaev}). 
\end{notation}
We will use the homological classes constructed above together with the specialisation of coefficients in order to prove the main result, which we remind below. 
\begin{thm}(Topological state sum model for the Witten-Reshetikhin-Turaev invariants)\\
Let $M$ be a closed oriented $3$-manifold and $L$ a framed oriented link with $l$ components such that $M$ is obtained by surgery along $L$. Let us choose a braid $\beta_n$ such that $L=\widehat{\beta_n}$. 
Now, for $i_1,...,i_{n} \in  \{0,...,\cN-2\}$, we consider the following Lagrangian intersection:
\begin{equation}
\begin{cases}
& \Lambda_{\bar{i}}(\beta_n) \in \Z[x_1^{\pm 1},...,x_l^{\pm 1},y_1^{\pm 1}..., y_l^{\pm 1}, d^{\pm 1}]\\
& \Lambda_{\bar{i}}(\beta_n):=\prod_{i=1}^l x_{C(p_i)}^{ \left(f_{p_i}-\sum_{j \neq {p_i}} lk_{p_i,j} \right)} \cdot \prod_{i=1}^n x_{C(i)}^{-1} 
  \   \left\langle(\beta_{n} \cup {\mathbb I}_{n+3l+1} ) \ { \color{red} \mathscr F_{\bar{i}}^{\cN}}, {\color{dgreen} \mathscr L_{\bar{i}}^{\cN}}\right\rangle. \end{cases}
\end{equation} 
Then the $\cN^{th}$ Witten-Reshetikhin-Turaev invariant has the following model:
\begin{equation}
\begin{aligned}
\tau  _{\cN}(M)=\frac{\{1\}^{-l}_{\xi}}{\cD^b \cdot \Delta_+^{b_+}\cdot \Delta_-^{b_-}}\cdot {\Huge{\sum_{i_1,..,i_n=0}^{\cN-2}}}   \left(\sum_{\substack{ \tiny 1 \leq N_1,...,N_l \leq \cN-1 \\ \bar{i}\in C(N_1,..,N_l)}}  \Lambda_{\bar{i}}(\beta_n) \Bigm| _{\psi^{C}_{\xi,N_1,...,N_l}}\right).
\end{aligned}
\end{equation} 
\end{thm}

\begin{proof} The Witten-Reshetikhin-Turaev invariant at level $\cN$ is defined using the coloured Jones polynomials of the link $L$ whose components are coloured with the Kirby colour $\Omega$ (\cite{Turaev},\cite{O}): 
\begin{equation}
\tau_{\cN}(M)=\frac{1}{\cD^b \cdot \Delta_+^{b_+}\cdot \Delta_-^{b_-}}\cdot J_{\Omega,...,\Omega}(L,\xi).
\end{equation}
This means that the invariant is given by the following linear combination of  coloured Jones polynomials, with colours less than $\cN-1$:
\begin{equation}
\begin{aligned}
\tau & _{\cN}(M)=\frac{1}{\cD^b \cdot \Delta_+^{b_+}\cdot \Delta_-^{b_-}} \cdot \sum_{1 \leq N_1,...,N_l \leq \cN-1 }[N_1]_{\xi}\cdot ... \cdot [N_l]_{\xi} \cdot J_{N_1,...,N_l}(L,\xi).
\end{aligned}
\end{equation} 
\subsection*{\bf Step I (WRT invariant as a sum of intersections in various configuration spaces)}

\

Now, we remind the topological formula for the coloured Jones polynomials, which is presented in Theorem \ref{THEOREM}:
\begin{equation}
\begin{aligned}
J_{N_1,...,N_l}(L,q)& =~ q^{ \sum_{i=1}^{l}\left( f_i- \sum_{j \neq i} lk_{i,j}\right)(N_i-1)} \cdot \\
 & \cdot \left(\sum_{\bar{i}\in C(\bar{N})} \left( \prod_{i=1}^{n}x^{-1}_{C(i)} \right)\cdot  \left\langle(\beta_{n} \cup {\mathbb I}_{n+1} ) \ { \color{red} \mathscr F_{\bar{i}}^{\CC}}, {\color{dgreen} \mathscr L_{\bar{i}}^{\CC}}\right\rangle \right)\Bigm| _{\psi^{C}_{q,N_1,...,N_l}}.
\end{aligned}
\end{equation}
We notice that the variables $x_{C(p_1)},...,x_{C(p_l)}$ correspond to the special strands of the braid $p_1,...,p_l$ which are all associated to different components of the link. More precisely, we have that:
$$(C_{p_1},...,C_{p_l})=(N_1,...,N_l)$$ as unordered families. 
We remind the notation $C_{p_i}=N_{C(p_i)}$. Further on, the variables are specialised in the following manner:
\begin{equation} 
 {\psi^{C}_{q,N_1,...,N_l}}(x_{C(p_i)})=q^{N_{C(p_i)}-1}=q^{C_{p_i}-1}, \forall i\in \{1,...,l\}.
\end{equation}
This remark allows us to encode the framing correction and we obtain the following formula:
\begin{equation}
\begin{aligned}
&J_{N_1,...,N_l}(L,q) =\\
& \hspace{-3mm}\left( \prod_{i=1}^l x_{C(p_i)}^{ \left(f_{p_i}-\sum_{j \neq {p_i}} lk_{p_i,j} \right)} \prod_{i=1}^n x_{C(i)}^{-1} 
  \cdot \sum_{\bar{i}\in C(\bar{N})} \left\langle(\beta_{n} \cup {\mathbb I}_{n+1} ) \ { \color{red} \mathscr F_{\bar{i}}^{\CC}}, {\color{dgreen} \mathscr L_{\bar{i}}^{\CC}}\right\rangle \right)\Bigm| _{\psi^{C}_{q,N_1,...,N_l}}.
\end{aligned}
\end{equation} 
(here, we used the notations from the statement of Theorem \ref{THEOREMW} concerning the framings).

This means that the $3-$manifold invariant is given by the expression presented below:
\begin{equation}\label{eq:p1}
\begin{aligned}
&\tau_{\cN}(M)= \frac{1}{\cD^b \cdot \Delta_+^{b_+}\cdot \Delta_-^{b_-}} \cdot \mathlarger{\mathlarger{\sum}}_{1 \leq N_1,...,N_l \leq \cN -1}[N_1]_{\xi}\cdot ... \cdot [N_l]_{\xi} \cdot\\
& \hspace{-5mm}\cdot \left( \prod_{i=1}^l x_{C(p_i)}^{ \left(f_{p_i}-\sum_{j \neq {p_i}} lk_{p_i,j} \right)} \prod_{i=1}^n x_{C(i)}^{-1} 
  \cdot \sum_{\bar{i}\in C(\bar{N})} \left\langle(\beta_{n} \cup {\mathbb I}_{n+1} ) \ { \color{red} \mathscr F_{\bar{i}}^{\CC}}, {\color{dgreen} \mathscr L_{\bar{i}}^{\CC}}\right\rangle \right)\hspace{-1mm}\Bigm| _{\psi^{C}_{\xi,N_1,...,N_l}}.
\end{aligned}
\end{equation} 
\subsection*{\bf Step II (Construction of homology classes in a fixed configuration space)}

\

In this part we concentrate on each intersection pairing which occurs in the above formula. They are given by:
\begin{equation}
\begin{aligned}
&\left\langle(\beta_{n} \cup {\mathbb I}_{n+1} ) \ { \mathscr F_{\bar{i}}^{\CC}}, {\mathscr L_{\bar{i}}^{\CC}}\right\rangle.\\
\end{aligned}
\end{equation}
This pairing comes from an intersection in the configuration space of $1+\sum_{i=1}^{n}(C_i-1)$ points in the $(2n+1)-$ punctured disk, and the homology classes belong to the homology groups:
$$\mathscr F_{\bar{i}}^{\CC} \in H^{-n}_{2n+1, {\scriptscriptstyle 1+\sum_{i=1}^{n}}(C_i-1)}; \ \ \ \ \ \ \mathscr L^{\CC}_{\bar{i}}\in H^{-n,\partial}_{2n+1,1+\sum_{i=1}^{n} (C_i-1)}.$$
We would like to arrive at an intersection in a configuration space where the number of particles does not depend on the individual components given by the set $\CC$. 

In order to achieve this, we use the property that all components of this set are bounded by the level of the $3-$manifold invariant, more precisely we have that:
$$0 \leq i_k\leq C_k-1 \leq \cN-2, \ \ \forall k \in \{1,...,n\}.$$

Now let us investigate the geometric supports of the two classes. For any $k$, we remark that the geometric support of $\mathscr F_{\bar{i}}^{\CC}$ has:
\begin{itemize}
\item[•] $i_k$ curves ending in the $k^{th}$ puncture
\item[•] $C_k-i_k-1$ curves ending in the $(2n-k+1)^{st}$ puncture.
\end{itemize}
Based on these remarks, we will ``complete'' each index which corresponds to a colour $C_k-1$ up to $\cN-2$. We will do this using the property that the action of the braid $(\beta_n \cup {\mathbb I}_{n+1})$ is trivial on the right hand side of the $(2n+1)$-punctured disk.  

For each $k \in \{1,...,n\}$ let us add $\cN-C_k-1$ extra segments/ configuration points on the part of the geometric supports of the classes $\mathscr F_{\bar{i}}^{\CC}$ and $\mathscr L_{\bar{i}}^{\CC}$ which end/ go around the puncture $2n-1-k$. Thanks to this change, each of the new geometric supports has in total $\cN-2$ curves/ configuration points which end/ go around symmetric punctures of the punctured disk. 

\begin{defn} (Level $\cN$ Homology classes)\\
Following this procedure, we consider the homology classes given by the geometric supports which are presented in figure \ref{Picture333}, and denote then by:
$${\color{red} F_{\bar{i}}^{\cN} \in H^{-n}_{2n+1,n(\cN-2)+1}} \ \ \ \ \ \ \ \ \  \ \ \ \ \text{ and }\ \ \ \ \ \ \ \ \ \ \ \  \  {\color{dgreen} L_{\bar{i}}^{\cN}\in H^{-n,\partial}_{2n+1,n(\cN-2)+1}}$$
\vspace{-8mm}
\begin{figure}[H]
\centering
\includegraphics[scale=0.4]{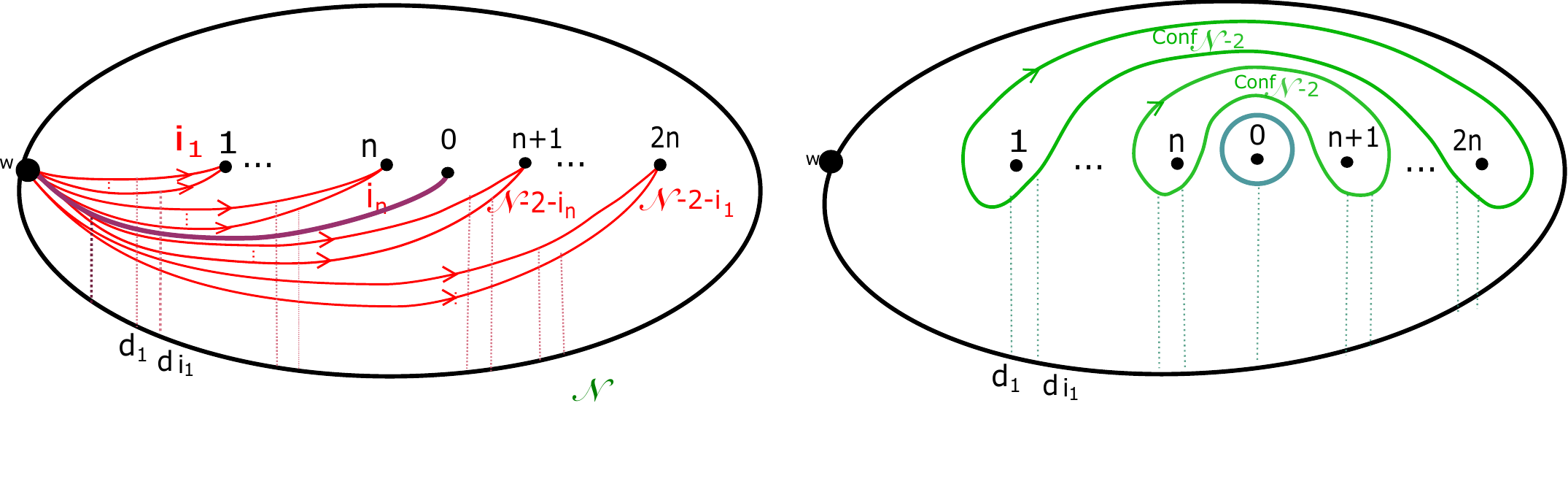}
\vspace{-5mm}
\caption{Classes corresponding to the level $\cN$ and multi-index $\bar{i}$ }
\label{Picture333}
\end{figure}
\end{defn}
Further on, we show that the change of the classes does not affect the outcome of the intersection pairing. 

\begin{prop}(Equality of intersection pairings in different configuration spaces)\label{eq:p2}\\
For any choice of indices $ \ 0 \leq i_k\leq C_k-1, k\in \{1,..,n\}$, we have the following relation between intersection pairings:
\begin{equation}
\begin{aligned}
\left\langle(\beta_{n} \cup {\mathbb I}_{n+1} ) \ { \mathscr F_{\bar{i}}^{\CC}}, {\mathscr L_{\bar{i}}^{\CC}}\right\rangle =& 
\left\langle(\beta_{n} \cup {\mathbb I}_{n+1} ) \ { F_{\bar{i}}^{\cN}}, { L_{\bar{i}}^{\cN}}\right\rangle.
\end{aligned}
\end{equation}
\end{prop}
\begin{proof}
This relation can be seen from the formula of the graded intersection form. The pairing is encoded by the intersection points between the geometric supports of the classes in the base configuration space, which are graded by certain coefficients coming from the local system. 

We denote the geometric support of a class $\cC$ by $s \hspace{0.5mm} \cC$. Further on, we notice that the following intersections in the configuration space:
\begin{equation}
\begin{aligned}
& 1) \left( (\beta_{n} \cup {\mathbb I}_{n+1} ) \ s{ \mathscr F_{\bar{i}}^{\CC}} \right) \cap {s \mathscr L_{\bar{i}}^{\CC}}\\
& 2) \left( (\beta_{n} \cup {\mathbb I}_{n+1} ) s{ F_{\bar{i}}^{\cN}}\right)\cap s{L_{\bar{i}}^{\cN}}
\end{aligned}
\end{equation}
have the same intersection points in the left hand side of the disk, and they differ by the fact that the second pair has more intersection points in the right hand side of the disk. Looking at the intersection in the configuration space in the punctured disk, this remark establishes a bijection between the intersection points from 1) and the intersection points from 2).

Let us fix an intersection point $P$ from 1) and denote by $\tilde{P}$ its correspondent in 2) .
Now, we look at the monomials which are associated to these points. The loop in the configuration space which corresponds to $\tilde{P}$ is obtained from the loop corresponding to $P$ union with another $$\sum_{k=1}^{n}(\cN-C_k-1)$$ loops which pass through the extra intersection points in the right hand side of the disk. However, we see that the extra loops are evaluated trivially by the local system since they do not twist or go around any puncture, and so they contribute with coefficients which are all $1$. This concludes that the two intersection pairings lead to the same result.
\end{proof}
Proposition \ref{eq:p2} together with formula \eqref{eq:p1} show that we can obtain WRT invariant from intersections between the new homology classes, as below:
\begin{equation}\label{eq:1'}  
\begin{aligned}
\tau_{\cN}(M)=& \frac{1}{\cD^b \cdot \Delta_+^{b_+}\cdot \Delta_-^{b_-}} \cdot \mathlarger{\mathlarger{\sum}}_{1 \leq N_1,...,N_l \leq \cN-1 }[N_1]_{\xi}\cdot ... \cdot [N_l]_{\xi} \cdot\\
& \hspace{-5mm}\cdot \left( \prod_{i=1}^l x_{C(p_i)}^{ \left(f_{p_i}-\sum_{j \neq {p_i}} lk_{p_i,j} \right)} \prod_{i=1}^n x_{C(i)}^{-1} 
  \cdot \sum_{\bar{i}\in C(\bar{N})} \left\langle(\beta_{n} \cup {\mathbb I}_{n+1} ) \ { \color{red} F_{\bar{i}}^{\cN}}, {\color{dgreen}  L_{\bar{i}}^{\cN}}\right\rangle \right)\Bigm| _{\psi^{C}_{\xi,N_1,...,N_l}}.
\end{aligned}
\end{equation} 

We arrived at a state sum model for $\tau_{\cN}$ as intersections between homology classes which are given by geometric supports in a fixed ambient manifold, namely the configuration space of $n(\cN-2)+1$ points on the $2n+1$ punctured disk. Then, the ``individual colours'' from the initial formula appear in the specialisations of coefficients and also in the quantum numbers coming from the Kirby colour. 

{\bf Encoding the coefficients of the Kirby colour} Pursuing this line, in the following parts we aim to understand geometrically the coefficients which come from the Kirby colour and encode them by intersections between the homology classes.

For the moment, we have an intersection in the $(2n+1)$-punctured disk $\cD_{2n+1}$ which takes values in the ring $\Z[x_1^{\pm 1},...,x_l^{\pm 1}, d^{\pm 1}]$ (definition \ref{D:4}). Let us look at the terms which appear in formula \eqref{eq:1'}. For a fixed set of colours $N_1,...,N_l$ we have a state sum which is given by:
 \begin{equation}
 \begin{aligned}
&[N_1]_{\xi}\cdot ... \cdot [N_l]_{\xi}  \cdot \left( \sum_{\bar{i}\in C(\bar{N})} \left\langle(\beta_{n} \cup {\mathbb I}_{n+1} ) \ { \color{red} F_{\bar{i}}^{\cN}}, {\color{dgreen}  L_{\bar{i}}^{\cN}}\right\rangle \right)\Bigm| _{\psi^{C}_{\xi,N_1,...,N_l}}=\\
& \sum_{\bar{i}\in C(\bar{N})}
[N_1]_{\xi}\cdot ... \cdot [N_l]_{\xi}  \cdot \left(  \left\langle(\beta_{n} \cup {\mathbb I}_{n+1} ) \ { \color{red} F_{\bar{i}}^{\cN}}, {\color{dgreen}  L_{\bar{i}}^{\cN}}\right\rangle \right)\Bigm| _{\psi^{C}_{\xi,N_1,...,N_l}}
\end{aligned}
\end{equation}

Now, we want to understand topologically the term 
\begin{equation}\label{eq:2}  
[N_1]_{\xi}\cdot ... \cdot [N_l]_{\xi}  \cdot \left(  \left\langle(\beta_{n} \cup {\mathbb I}_{n+1} ) \ { \color{red} F_{\bar{i}}^{\cN}}, {\color{dgreen}  L_{\bar{i}}^{\cN}}\right\rangle \right)\Bigm| _{\psi^{C}_{\xi,N_1,...,N_l}}
\end{equation}
in an unified way which does not depend on the choice of individual colours $N_1,..,N_l$. More precisely, we would like to see this term as a $\psi^{C}_{\xi,N_1,...,N_l}$ specialisation of an intersection which does not depend on the individual colours. 

\subsection*{\bf Step III (Add extra punctures to the punctured disk)}

We do this by adding $3l$ points to our punctured disk and work in $\cD_{2n+3l+1}$. In this manner we have a richer local system which carries monodromies around these additional punctures.
\begin{defn} (Homology classes using the $(2n+3l+1)$-punctured disk)\\
We consider the homology classes given by the geometric supports which are presented in figure \ref{Picture33}, in the configuration space in the $(2n+3l+1)$-punctured disk:
$${\color{red} F_{\bar{i}}^{\cN,l} \in H^{-n}_{2n+1,n(\cN-2)+1,l}} \ \ \ \ \ \ \ \ \  \ \ \ \ \text{ and }\ \ \ \ \ \ \ \ \ \ \ \  \  {\color{dgreen} L_{\bar{i}}^{\cN,l}\in H^{-n,\partial}_{2n+1,n(\cN-2)+1,l}}$$
\vspace{-8mm}
\begin{figure}[H]
\centering
\includegraphics[scale=0.25]{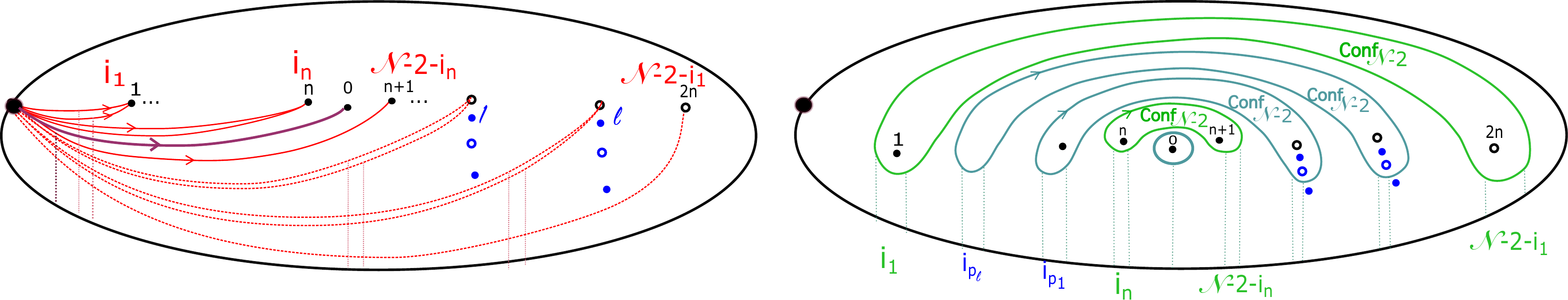}
\vspace{-5mm}
\caption{Homology Classes from the disk with extra punctures}
\label{Picture33}
\end{figure}
\end{defn}

We remind that the homologies $H^{-n}_{2n+1,n(\cN-2)+1,l}$ and $H^{-n,\partial}_{2n+1,n(\cN-2)+1,l}$ are modules over \\
$\Z[x_1^{\pm 1},...,x_l^{\pm 1},y_1^{\pm 1}..., y_l^{\pm 1}, d^{\pm 1}]$ and in the next part we will use the new variables $y_1,...,y_l$.  We remember that the monodromy of $\Phi$ around the extra blue punctures corresponding to $\bar{p}_1,...,\bar{p}_l$ is evaluated through $f_{C}$ with the variables (using definition \ref{Localsyst} and relation \eqref{eq:8}):
\begin{equation} 
\left( \begin{array}{c}
y^{2}_1 \\
y^{-2}_1 \\
y^{-1}_1
\end{array} \right)
,\ldots,
\left(\begin{array}{c}
y^{2}_l \\
y^{-2}_l \\
y^{-1}_l
\end{array}\right).
\end{equation}
The order of these evaluations might vary, but the columns correspond exactly to the above triples. Then, these monodromies get evaluated through $\psi^{C}_{\xi,N_1,...,N_l}$ to the following values:
\begin{equation} 
\left( \begin{array}{c}
\xi^{2N_1} \\
\xi^{-2N_1} \\
\xi^{-N_1}
\end{array} \right)
,\ldots,
\left(\begin{array}{c}
\xi^{2N_l} \\
\xi^{-2N_l} \\
\xi^{-N_l}
\end{array}\right).
\end{equation}
 Based on this remark, we notice that we can encode the quantum numbers using the variables $y_1,...,y_l$ and we have following relation:
\begin{equation} \label{eq:3'}
 [N_1]_{\xi}\cdot ... \cdot [N_l]_{\xi}=\{1\}^{-l}_{\xi} \left( \prod_{i=1}^{l} (y_i-y_i^{-1}) \right)|_{\psi^C_{\xi,N_1,...,N_l}}.
\end{equation}
\begin{lem}
Using this property, we obtain a new intersection pairing and the following relation holds:
 \begin{equation}\label{eq:3}  
 \begin{aligned}
& [N_1]_{\xi}\cdot ... \cdot [N_l]_{\xi}  \cdot \left( \left\langle(\beta_{n} \cup {\mathbb I}_{n+1} ) \ { \color{red} F_{\bar{i}}^{\cN}}, {\color{dgreen}  L_{\bar{i}}^{\cN}}\right\rangle \right)\Bigm| _{\psi^{C}_{q,N_1,...,N_l}}=\\
& =\{1\}^{-l}_{\xi} \left( \prod_{i=1}^{l} (y_i-y_i^{-1}) \left\langle(\beta_{n} \cup {\mathbb I}_{n+3l+1} ) \ { \color{red} F_{\bar{i}}^{\cN,l}}, {\color{dgreen}  L_{\bar{i}}^{\cN,l}}\right\rangle \right)\Bigm| _{\psi^{C}_{q,N_1,...,N_l}}
 \end{aligned}
\end{equation}
\end{lem}
\begin{proof}
Following equation \eqref{eq:3'} we have the equality of the coefficients which appear in both terms from above. Now, we notice that the addition of the extra punctures does not change the intersection forms, and so we have:
\begin{equation}
\left\langle(\beta_{n} \cup {\mathbb I}_{n+1} ) \ { \color{red} F_{\bar{i}}^{\cN}}, {\color{dgreen}  L_{\bar{i}}^{\cN}}\right\rangle=\left\langle(\beta_{n} \cup {\mathbb I}_{n+3l+1} ) \ { \color{red} F_{\bar{i}}^{\cN,l}}, {\color{dgreen}  L_{\bar{i}}^{\cN,l}}\right\rangle
\end{equation}
We can see this from the fact that the supports of the homology classes have the same intersection points, and so the only question we have concerns the gradings which they carry. The second intersection belongs to  a covering where there could be potential monodromies around the $3l$-blue punctures. However, we remark that the loops which are associated to the intersection points in the right hand side of the disk do not wind around the blue punctures and so they have trivial monodromies.
\end{proof}
 Using this Lemma together with the formula from equation \eqref{eq:1}, we conclude the following formula:
\begin{equation}\label{eq:6}
\begin{aligned}
&\tau_{\cN}(M)= \frac{\{1\}^{-l}_{\xi}}{\cD^b \cdot \Delta_+^{b_+}\cdot \Delta_-^{b_-}} \cdot \sum_{1 \leq N_1,...,N_l \leq \cN-1 } \\
& \hspace{-7mm}\hspace{-0.5mm} \left( \prod_{i=1}^l x_{C(p_i)}^{ \left(f_{p_i}-\sum_{j \neq {p_i}} lk_{p_i,j} \right)} \prod_{i=1}^n x_{C(i)}^{-1} 
  \cdot \prod_{i=1}^{l} (y_i-y_i^{-1}) \sum_{\bar{i}\in C(\bar{N})} \left\langle(\beta_{n} \cup {\mathbb I}_{n+3l+1} ) \ { \color{red} F_{\bar{i}}^{\cN,l}}, {\color{dgreen}  L_{\bar{i}}^{\cN,l}}\right\rangle \right)\hspace{-1.5mm}\Bigm| _{\psi^{C}_{\xi,N_1,...,N_l}}.
\end{aligned}
\end{equation} 
We remark that we arrived at an expression given by the graded intersections (the terms between the brackets), which do not depend anymore on the choice of colours $N_1,..,N_l$. After that we have to specialise them using the change of coefficients $\psi^{C}_{\xi,N_1,...,N_l}$.
\subsection*{\bf Step IV (Coefficients of the Kirby colour encoded by circles in the supports of the homology classes)}

\

In the last part, we will show that we can encode the coefficients of the Kirby colour by adding $l$ points to our configuration space and considering the classes which are obtained from the supports of the classes $F_{\bar{i}}^{\cN,l}$ and $L_{\bar{i}}^{\cN,l}$ by adding $l$ extra circles.

More specifically, we prove that the pairing that arises from the homology classes $\mathscr F_{\bar{i}}^{\cN}$ and $\mathscr L_{\bar{i}}^{\cN}$ captures precisely the extra coefficients from equation \eqref{eq:3}. We remind that:
$$\mathscr F_{\bar{i}}^{\cN} \in H^{-n}_{2n+1,n(\cN-2)+l+1,l} \ (\text{figure } \ref{Picture}) \ \ \ \ \ \ \  \mathscr L_{\bar{i}}^{\cN} \in H^{-n,\partial}_{2n+1,n(\cN-2)+l+1,l} \  (\text{figure } \ref{Picture2}).$$

\begin{prop} (Encoding the Kirby colour)\label{P:1} For any index $\bar{i}$ we have: 
 \begin{equation} \label{eq:5}
\left( \prod_{k=1}^{l} (y_k-y_k^{-1}) \right) \left\langle(\beta_{n} \cup {\mathbb I}_{n+3l+1} ) \ { \color{red} F_{\bar{i}}^{\cN,l}}, {\color{dgreen}  L_{\bar{i}}^{\cN,l}}\right\rangle=
 \left\langle(\beta_{n} \cup {\mathbb I}_{n+3l+1} ) \ { \color{red} \mathscr F_{\bar{i}}^{\cN}}, {\color{dgreen}  \mathscr L_{\bar{i}}^{\cN}}\right\rangle.
\end{equation}
\end{prop}
\begin{proof} For the computation of these intersections we will use the formulas for the intersection pairing presented in equation \ref{eq:1} and remark \ref{orientd}, where for each intersection point we count the product of local orientations in the disk and multiply it with the evaluation of the local system $\bar{\Phi}$ on the associated loop in the configuration space.

We notice that the classes which lead to these intersection pairings have similar geometric supports except the fact that:
\begin{itemize}
\item $s\mathscr F_{\bar{i}}^{\cN}$ is constructed from $s F_{\bar{i}}^{\cN,l}$ by adding $l$ extra circles
\item $s\mathscr L_{\bar{i}}^{\cN}$ comes from $s L_{\bar{i}}^{\cN,l}$ but it has $l$ extra points, one on each circle which goes around the punctures $p_k$ for $k\in\{1,...,l\}$.
\end{itemize}
Now, we look at the intersection points between the geometric supports which are obtained after we act with the braid:
\begin{equation}
\begin{aligned}
& 1) \left( (\beta_{n} \cup {\mathbb I}_{n+3l+1} ) s{ F_{\bar{i}}^{\cN,l}}\right)\cap s{L_{\bar{i}}^{\cN,l}}\\
& 2) \left( (\beta_{n} \cup {\mathbb I}_{n+3l+1} ) s{ \mathscr F_{\bar{i}}^{\cN}}\right)\cap s{\mathscr L_{\bar{i}}^{\cN}}.
\end{aligned}
\end{equation}
These two intersections have the same components in the left hand side of the punctured disk. The difference occurs in the right hand side of it. 

Let us denote by:
$$(q_1,r_1),...,(q_l,r_l)$$ the intersection points between the purple circles from $s{ \mathscr F_{\bar{i}}^{\cN}}$ and the blue disks from $s{\mathscr L_{\bar{i}}^{\cN}}$ which intersect them (they are the orange points from figure \ref{Picture3}). 

Also, let us look at the first pairing $1)$ and consider the set of intersection points between the supports of the homology classes:
$$\left((\beta_{n} \cup {\mathbb I}_{n+1} ) s{ F_{\bar{i}}^{\cN,l}}\right)\cap s{L_{\bar{i}}^{\cN,l}}=\{ \bar{m}_1,...,\bar{m}_s \}.$$ 
Here, each element is a multipoint in the configuration space, which has $n(\cN-2)+1$ components. 
\begin{rmk}
The set of intersection points between the new classes (from $2)$) is obtained from the above intersection points together with a choice of $l$ orange points which belong to different circles:
\begin{equation}\label{eq:4}
\left((\beta_{n} \cup {\mathbb I}_{n+3l+1} ) s{ \mathscr F_{\bar{i}}^{\cN}}\right)\cap s{ \mathscr L_{\bar{i}}^{\cN}}=\{ \bar{m}_1,...,\bar{m}_s \}\times \{q_1,r_1\}\times ... \times \{ q_l,r_l \}.
\end{equation}
\end{rmk}

\begin{figure}[H]
\centering
\includegraphics[scale=0.39]{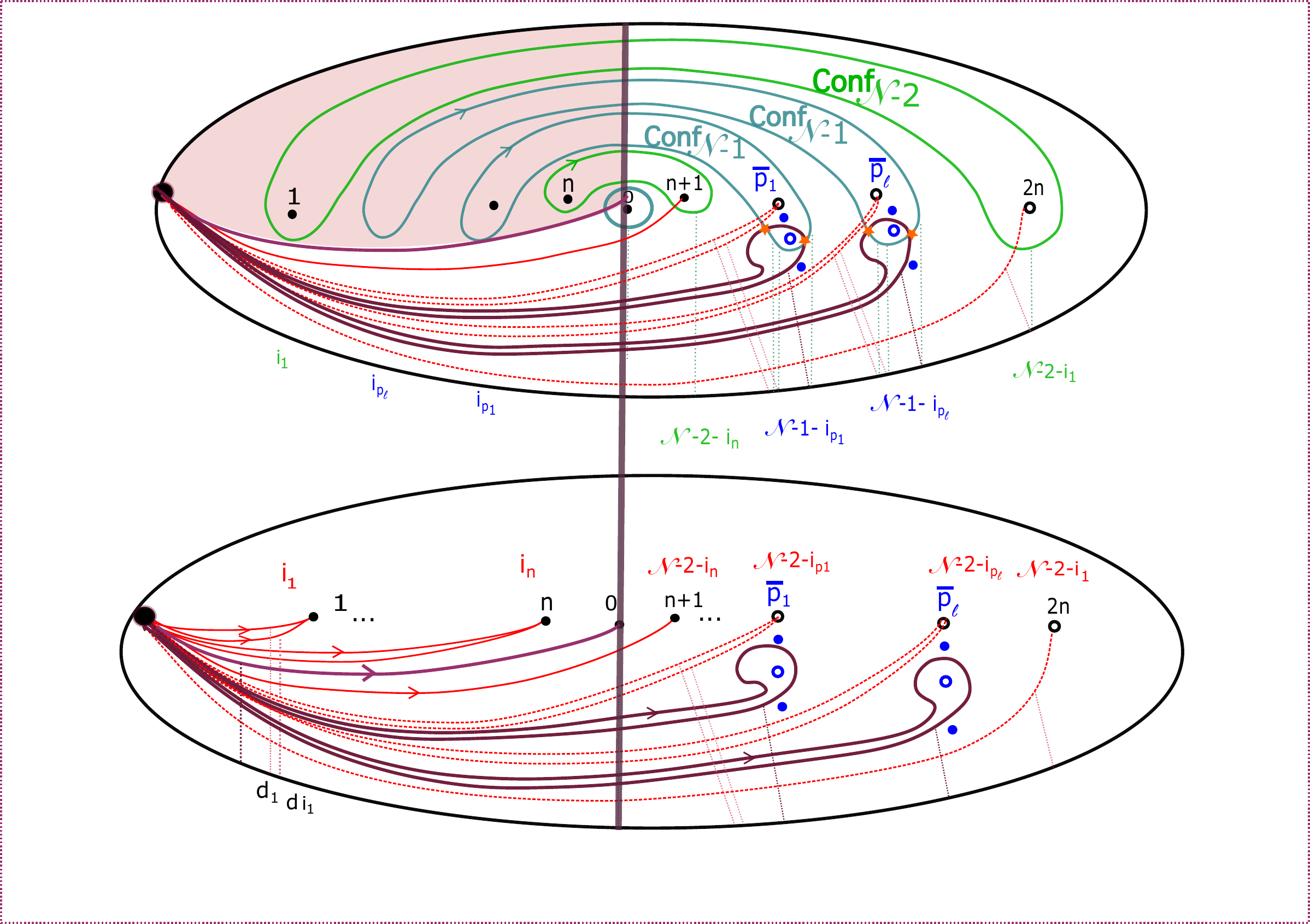}
\vspace{-15mm}
\caption{}
\label{Picture3}
\end{figure}
\begin{proof}
Let $P$ be a multipoint which belongs to the intersection $$\left((\beta_{n} \cup {\mathbb I}_{n+3l+1} ) s{ \mathscr F_{\bar{i}}^{\cN}}\right)\cap s{ \mathscr L_{\bar{i}}^{\cN}}.$$ This means that it has a exactly one component on each red curve and purple curve from $(\beta_{n} \cup {\mathbb I}_{n+3l+1} ) s{ \mathscr F_{\bar{i}}^{\cN}}$. In particular, it has exactly one point on each purple circle, which should be at the intersection with the dual support $s{ \mathscr L_{\bar{i}}^{\cN}}$. We notice that for any fixed $k\in\{1,...,l\}$ the $k^{th}$ purple circle intersects only one component from the dual support, given by the configuration space of $\cN-1$ points on the blue circle, in exactly two orange points: $\{q_k,r_k\}$. This means that $P$ has exactly one component from each of these sets with two elements. Then, for the rest of the points we should use the configuration space of $\cN-2$ points on the blue circles. The rest of the red curves intersected with the configuration spaces of $\cN-2$ particles on the circles give precisely an intersection point belonging to $$\left((\beta_{n} \cup {\mathbb I}_{n+1} ) s{ F_{\bar{i}}^{\cN,l}}\right)\cap s{L_{\bar{i}}^{\cN,l}}.$$ This procedure establishes the desired bijection. 
\end{proof}
So far, we saw the correspondence at the level of sets. Now we are interested in the coefficients coming from the local system. For this, we turn our attention towards the coefficients which are carried by the orange points. For each $k\in \{1,...,l\}$, we have look at the chosen point on the purple circle ($q_k$ or $r_k$) and to evaluate the monodromy of the yellow path corresponding this point around the punctures of the disk. These two paths are presented in picture \ref{Picture4}. 
\vspace{-2mm}
\begin{figure}[H]
\centering
\begin{subfigure}{7.5cm}
\includegraphics[scale=0.4]{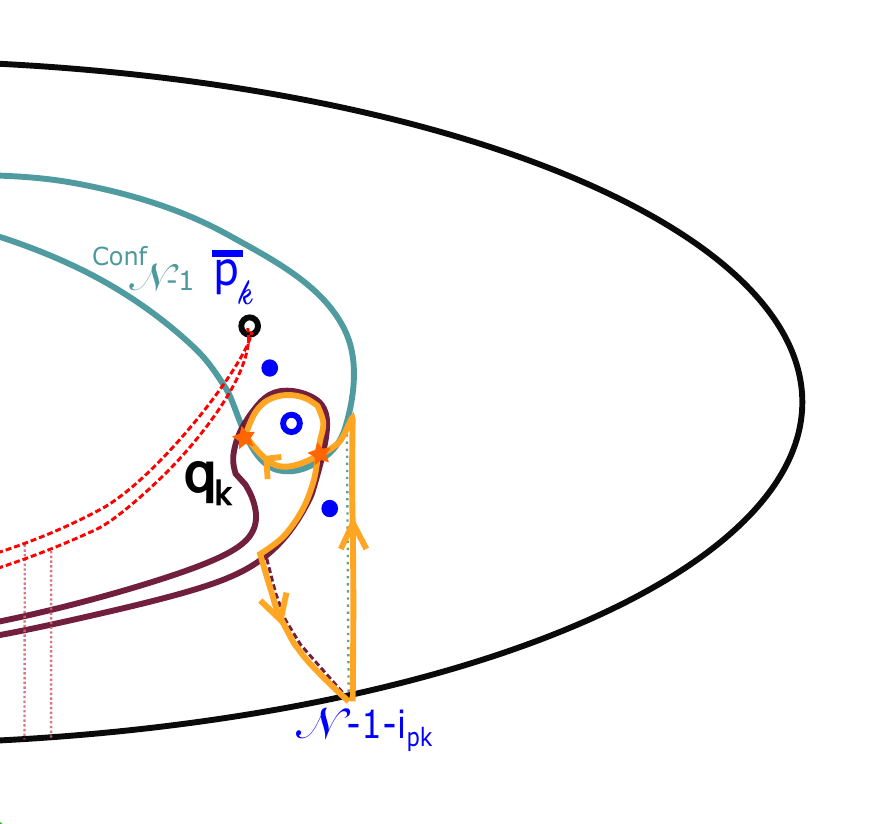}
\end{subfigure}
\begin{subfigure}{6cm}
\includegraphics[scale=0.4]{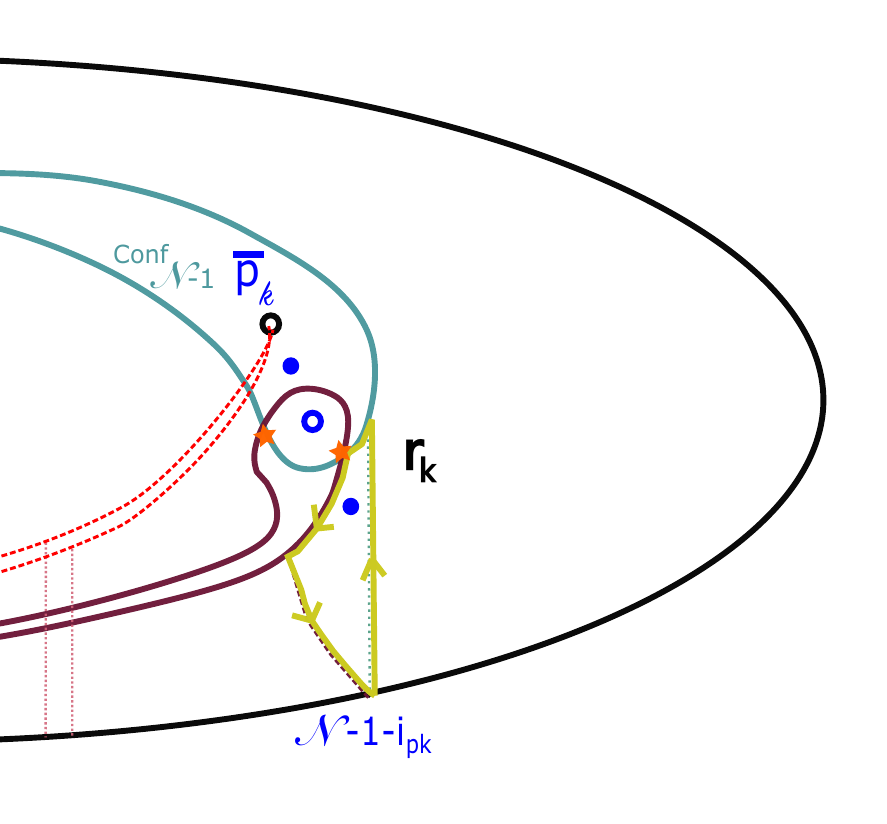}
\end{subfigure}
\caption{Paths corresponding to the intersection points}
\label{Picture4}
\end{figure}
We remind that the counter-clockwise monodromies around the three blue points are evaluated by the variables:
\begin{equation}
\left(\begin{array}{c}
y^{2}_k \\
y^{-2}_k \\
y^{-1}_k
\end{array}\right).
\end{equation}
Using this for the two paths from the picture, we see that the points $\{q_k,r_k\}$ carry the following coefficients:
\begin{equation}
\begin{aligned}
&(q_k) \ \ \ \ \  y^{2}_k\cdot y^{-1}_k=y_k\\
&(r_k) \ \ \ \ \ \ \ \ \ \ \ \ \ \   -y^{-1}_k.
\end{aligned}
\end{equation}
The yellow loops associated to the orange intersection points will not add extra $d$ components when we evaluate a loop corresponding to an intersection point in the configuration space. This is because these loops to not contribute to the relative winding in the configuration space. Also, the opposite sign comes from the local intersections in these two orange points.

We conclude that for any $k\in \{1,...,l\}$ the two points $\{q_k, r_k\}$ contribute to the grading with the coefficients $\{y_k, -y^{-1}_k\}$.
 This property together with the correspondence presented in relation \eqref{eq:4} shows that the extra coefficients that appear in the second intersection are precisely $$\prod_{k=1}^{l} (y_k-y_k^{-1})$$ which concludes the relation between the intersection pairings from \eqref{eq:5}.
\end{proof}
Now, using this property together with the expression from equation \eqref{eq:6}, we obtain:
\begin{equation}
\begin{aligned}
&\tau_{\cN}(M)= \frac{\{1\}^{-l}_{\xi}}{\cD^b \cdot \Delta_+^{b_+}\cdot \Delta_-^{b_-}} \cdot \sum_{1 \leq N_1,...,N_l \leq \cN-1 } \\
& \left( \prod_{i=1}^l x_{C(p_i)}^{ \left(f_{p_i}-\sum_{j \neq {p_i}} lk_{p_i,j} \right)} \prod_{i=1}^n x_{C(i)}^{-1} 
  \cdot  \sum_{\bar{i}\in C(\bar{N})} \left\langle(\beta_{n} \cup {\mathbb I}_{n+3l+1} ) \ { \color{red}\mathscr F_{\bar{i}}^{\cN}}, {\color{dgreen}  \mathscr L_{\bar{i}}^{\cN}}\right\rangle \right)\Bigm| _{\psi^{C}_{\xi,N_1,...,N_l}}=\\
& =\frac{\{1\}^{-l}_{\xi}}{\cD^b \cdot \Delta_+^{b_+}\cdot \Delta_-^{b_-}} \cdot \sum_{1 \leq N_1,...,N_l \leq \cN-1 } \\
&  \sum_{\bar{i}\in C(\bar{N})} \left( \prod_{i=1}^l x_{C(p_i)}^{ \left(f_{p_i}-\sum_{j \neq {p_i}} lk_{p_i,j} \right)} \prod_{i=1}^n x_{C(i)}^{-1} 
  \cdot  \left\langle(\beta_{n} \cup {\mathbb I}_{n+3l+1} ) \ { \color{red}\mathscr F_{\bar{i}}^{\cN}}, {\color{dgreen}  \mathscr L_{\bar{i}}^{\cN}}\right\rangle \right)\Bigm| _{\psi^{C}_{\xi,N_1,...,N_l}}.  
\end{aligned}
\end{equation}   
 Exchanging the two sums and taking care of the conditions which the colouring imposes on the multi-indices, we obtain the following formula:
\begin{equation}
\begin{aligned}
  &\tau_{\cN}(M) = \frac{\{1\}^{-l}_{\xi}}{\cD^b \cdot \Delta_+^{b_+}\cdot \Delta_-^{b_-}}\cdot {\Huge{\sum_{i_1,...,i_n=0}^{\cN-2}}} \\
 & \cdot \sum_{\substack{\bar{N}=(N_1,...,N_l) \\ 1 \leq N_1,...,N_l \leq \cN-1 \\ \bar{i}\in C(\bar{N})}}  \left(  \prod_{i=1}^l x_{C(p_i)}^{ \left(f_{p_i}-\sum_{j \neq {p_i}} lk_{p_i,j} \right)} \prod_{i=1}^n x^{-1}_{i} 
  \cdot   \left\langle(\beta_{n} \cup {\mathbb I}_{n+3l+1} ) \ { \color{red} \mathscr F_{\bar{i}}^{\cN}}, {\color{dgreen} \mathscr L_{\bar{i}}^{\cN}}\right\rangle \right)\Bigm| _{\psi^{C}_{\xi,N_1,...,N_l}}=\\
  & =\frac{\{1\}^{-l}_{\xi}}{\cD^b \cdot \Delta_+^{b_+}\cdot \Delta_-^{b_-}}\cdot {\Huge{\sum_{i_1,..,i_n=0}^{\cN-2}}} \left( \sum_{\substack{ \tiny 1 \leq N_1,...,N_l \leq \cN-1 \\ \bar{i}\in C(N_1,..,N_l)}}  \Lambda_{\bar{i}}(\beta_n) \Bigm| _{\psi^{C}_{\xi,N_1,...,N_l}}\right).
\end{aligned}
\end{equation} 
This relation concludes the proof of the main formula. 
\end{proof}

\begin{coro}(Detailed formula for the invariant)
\begin{equation}
\begin{aligned}
\tau & _{\cN}(M)=\frac{\{1\}^{-l}_{\xi}}{\cD^b \cdot \Delta_+^{b_+}\cdot \Delta_-^{b_-}}\cdot {\Huge{\sum_{i_1,...,i_n=0}^{\cN-2}}} \\
 & \cdot \sum_{\substack{\bar{N}=(N_1,...,N_l) \\ 1 \leq N_1,...,N_l \leq \cN-1 \\ \bar{i}\in C(\bar{N})}}  \left(  \prod_{i=1}^l x_{C(p_i)}^{ \left(f_{p_i}-\sum_{j \neq {p_i}} lk_{p_i,j} \right)} \prod_{i=1}^n x^{-1}_{C(i)} 
  \cdot   \left\langle(\beta_{n} \cup {\mathbb I}_{n+3l+1} ) \ { \color{red} \mathscr F_{\bar{i}}^{\cN}}, {\color{dgreen} \mathscr L_{\bar{i}}^{\cN}}\right\rangle \right)\Bigm| _{\psi^{C}_{\xi,N_1,...,N_l}}.
\end{aligned}
\end{equation} 
\end{coro}
\section{Topological model for the WRT invariants of $3$-manifols obtained as surgeries along knots}\label{S:6}
This section is devoted to the topological model presented above, for the particular case where the link is actually a knot (this means that $l=1$). Let us consider a knot $K$ which is the closure of a braid with $n$ strands $\beta_n \in B_n$.

In this case, we work in the covering of the configuration space of $n(\cN-2)+2$ particles in the punctured disk with $2n+4$ punctures, and we will use the homology groups:
$$H^{-n}_{2n+1, n(\cN-2)+2,1} \ \ \ \text{ and }\ \ \ \ H^{-n,\partial}_{2n+1,n(\cN-2)+2,1} \ \ \ \text{ as } \Z[x^{\pm 1},y^{\pm 1}, d^{\pm 1}]- \text{modules}.$$
This means that we have $3$ privileged blue points in the punctured disk. \subsection{Homology classes}
\begin{defn} a) (First Homology class) For any set $i_1,...,i_{n} \in \{0,...,\cN-2\}$ we consider the class given by the geometric support from the figure below:
$${\color{red} \mathscr F_{\bar{i}}^{\cN} \in H^{-n}_{2n+1,n(\cN-2)+2,1}}$$
\vspace{-10mm}
\begin{figure}[H]
\centering
\includegraphics[scale=0.3]{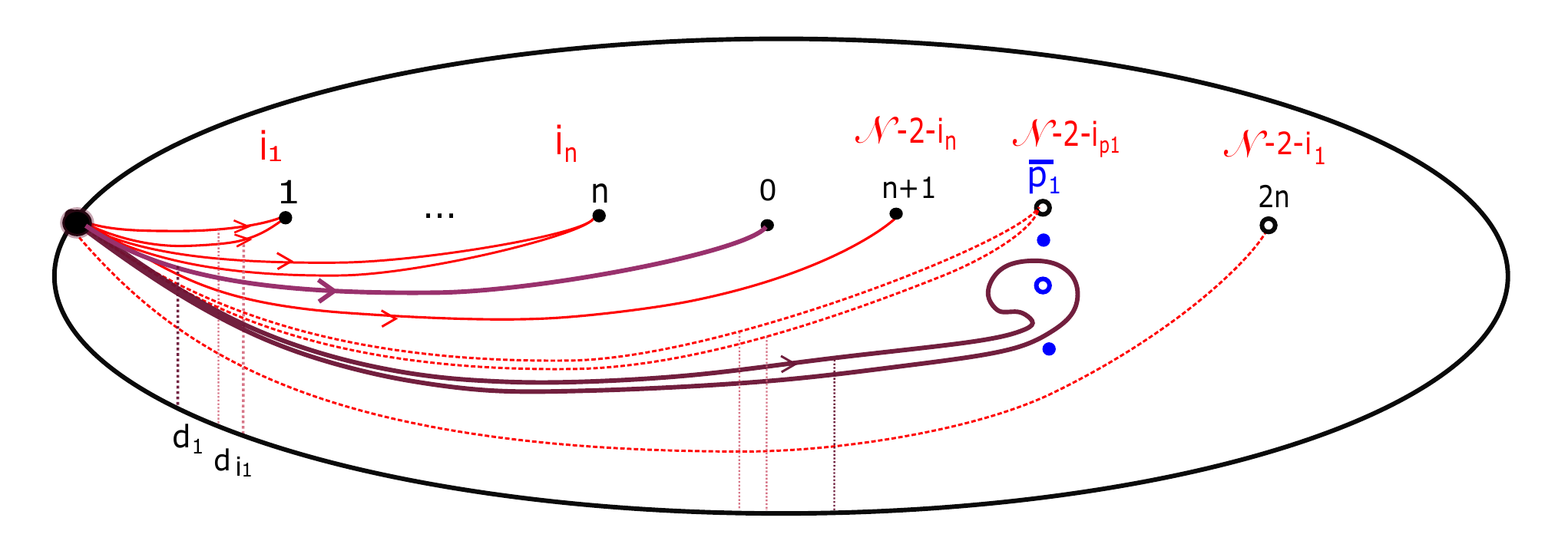}
\caption{}
\end{figure}
b) (Second Homology class) The second homology class is given by the following geometric support:
 $${\color{dgreen} \mathscr L_{\bar{i}}^{\cN} \in H^{-n,\partial}_{2n+1,n(\cN-2)+2,1}}$$
\vspace{-10mm}
\begin{figure}[H]
\centering
\includegraphics[scale=0.3]{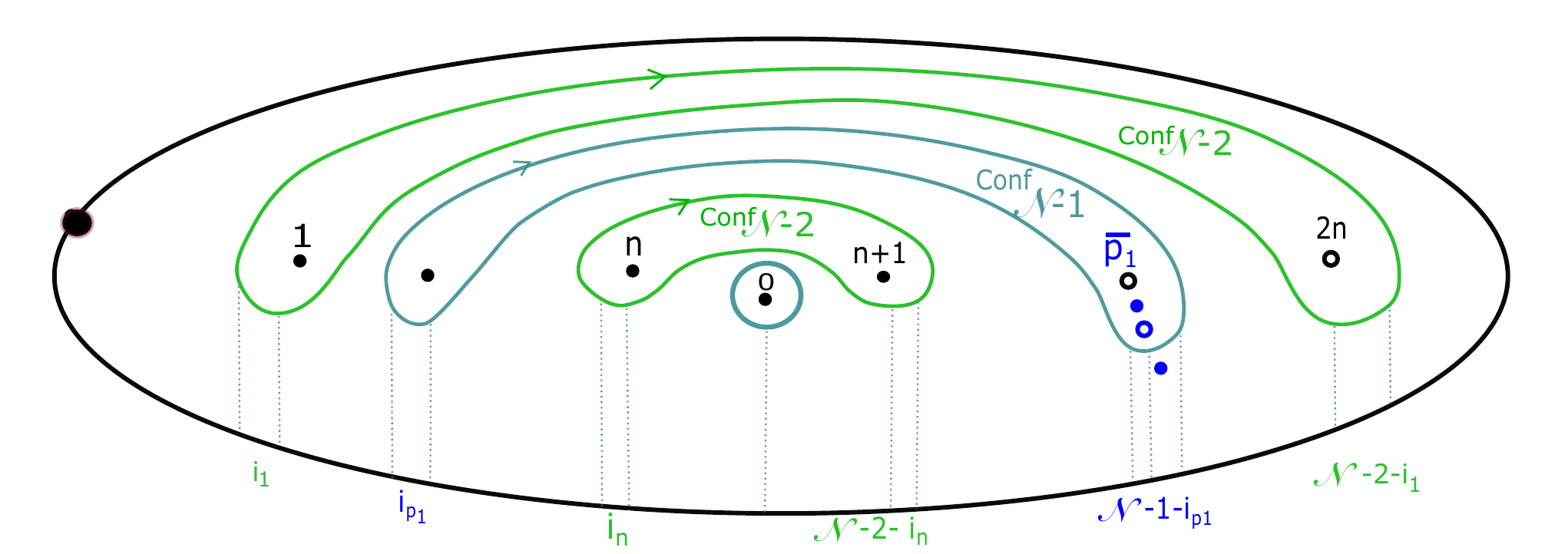}
\caption{}
\end{figure}
\end{defn}
Further on, we will use the specialisation of coefficients which corresponds to a coloring with one color $N \in \N$, which is given by:
$$ \psi^{C}_{q,N}: \Z[x^{\pm 1},y^{\pm 1},d^{\pm 1}] \rightarrow \Z[q^{\pm 1}]$$
\begin{equation}
\begin{cases}
&\psi^{C}_{q,N}(x)=q^{N-1},\\
&\psi^{C}_{q,N}(y_i)=q^{N}\\
&\psi^{C}_{q,N}(d)=q^{-2}.
\end{cases}
\end{equation} 

\begin{coro}(Topological model for the Witten-Reshetikhin-Turaev invariants of knot surgeries)\\
Let $M$ be a closed oriented $3$-manifold obtained by surgery along a knot $K$ with framing $f \in \Z$. We choose a braid $\beta_n$ such that $K=\widehat{\beta_n}$. Further on, for $i_1,...,i_{n} \in  \{0,...,\cN-2\}$, we consider the following Lagrangian intersection:
\begin{equation}
\begin{cases}
& \Lambda_{\bar{i}}(\beta_n) \in \Z[x^{\pm 1},y^{\pm 1}, d^{\pm 1}]\\
& \Lambda_{\bar{i}}(\beta_n):=x^{ f-w(\beta_n)} \cdot x^{-n} 
  \   \left\langle(\beta_{n} \cup {\mathbb I}_{n+4} ) \ { \color{red} \mathscr F_{\bar{i}}^{\cN}}, {\color{dgreen} \mathscr L_{\bar{i}}^{\cN}}\right\rangle \end{cases}.
\end{equation} 
Here, $w(\beta_n)$ is the writhe of the braid.
Then the $\cN^{th}$ Witten-Reshetikhin-Turaev invariant is obtained from these intersections as below:
\begin{equation}
\begin{aligned}
\tau  _{\cN}(M)=\frac{\{1\}^{-1}_{\xi}}{\cD^b \cdot \Delta_+^{b_+}\cdot \Delta_-^{b_-}}\cdot {\Huge{\sum_{i_1,..,i_n=0}^{\cN-2}}}   \left(\sum_{N=\text{max}\{i_1+1,...,i_n+1\}}^{\cN-1}  \Lambda_{\bar{i}}(\beta_n) \Bigm| _{\psi^{C}_{\xi,N}}\right).
\end{aligned}
\end{equation} 
\end{coro}

\begin{rmk}This tells as that the level $\cN$ WRT invariant of a surgery along a knot which is the closure of a braid with $n$ strands is obtained from states of graded intersections in the configuration space of $n(\cN-2)+2$ points in the $(2n+4)-$punctured disk.
\end{rmk}

\noindent {\itshape Mathematical Institute, University of Oxford, Woodstock Road, Oxford, OX2 6GG, United Kingdom}

\noindent {\tt palmeranghel@maths.ox.ac.uk}

\noindent \href{http://www.cristinaanghel.ro/}{www.cristinaanghel.ro}

\end{document}